\documentclass{article}

\usepackage{arxiv}
\usepackage{amsmath}
\usepackage{graphicx}
\usepackage{doi}
\usepackage{longtable}
\usepackage{amsmath}
\usepackage{graphicx}
\usepackage{geometry}
\usepackage{color}
\usepackage{algorithm}
\usepackage{algpseudocode}
\usepackage{amsthm}
\usepackage{bbm}
\usepackage{amssymb}
\usepackage{appendix}
\newtheorem{proposition}{Proposition}
\newtheorem{theorem}{Theorem}
\title{Who Goes Next? Optimizing the Allocation of Adherence-Improving Interventions}

\date{July 9, 2024}	

\author{ \href{https://orcid.org/0000-0003-2404-1635}{\includegraphics[scale=0.06]{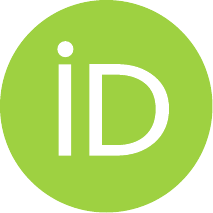}\hspace{1mm}Daniel Otero-Leon}\thanks{Webpage: \href{https://www.danieloteroleon.com}{www.danieloteroleon.com}} \\
	Department of Systems and Information Engineering\\
	University of Virginia\\
	Charlottesville, VA 22903 \\
	\texttt{dfotero@virginia.edu} \\
	\And
	\hspace{1mm}Mariel S. Lavieri \\
	Department of Industrial and Operations Engineering\\
	University of Michigan\\
	Ann Arbor, MI 48109 \\
	\texttt{lavieri@umich.edu} \\
 \And
	\hspace{1mm}Brian T. Denton \\
	Department of Industrial and Operations Engineering\\
	University of Michigan\\
	Ann Arbor, MI 48109 \\
	\texttt{btdenton@umich.edu} \\
  \And
	\hspace{1mm}Jerremy Sussman \\
	Department of Internal Medicine\\
	University of Michigan\\
	Ann Arbor, MI 48109 \\
	\texttt{jeremysu@med.umich.edu} \\
  \And
	\hspace{1mm}Rodney A. Hayward \\
	Department of Internal Medicine\\
	University of Michigan\\
	Ann Arbor, MI 48109 \\
	\texttt{rhayward@med.umich.edu} \\
}

\renewcommand{\shorttitle}{\textit{arXiv} Template}

\hypersetup{
pdftitle={A template for the arxiv style},
pdfsubject={q-bio.NC, q-bio.QM},
pdfauthor={David S.~Hippocampus, Elias D.~Striatum},
pdfkeywords={First keyword, Second keyword, More},
}

\begin{document}
\maketitle

\begin{abstract}
	Long-term adherence to medication is a critical factor in preventing chronic diseases, such as cardiovascular disease. To address poor adherence, physicians may recommend adherence-improving interventions; however, such interventions are costly and limited in their availability. Knowing which patients will stop adhering helps distribute the available resources more effectively. We developed a binary integer program (BIP) model to select patients for adherence-improving intervention under budget constraints.  We further studied a long-term adherence prediction model using dynamic logistic regression (DLR) model that uses patients' claim data, medical health factors, demographics, and monitoring frequencies to predict the risk of future non-adherence. We trained and tested our predictive model to longitudinal data for cardiovascular disease in a large cohort of patients taking medication for cholesterol control seen in the national Veterans Affairs health system. Our study shows the importance of including past adherence to increase prediction accuracy. Finally, we assess the potential benefits of using the prediction model by proposing an algorithm that combines the DLR and BIP models to decrease the number of CVD events in a population.
\end{abstract}

\keywords{Cardiovascular diseases \and Medication adherence forecast \and Resource allocation}

\section{Introduction}

Correct long-term use of medications is one of the primary methods to prevent chronic diseases effectively \cite{Burnier2018}. However, studies have shown that many patients will eventually stop taking medications at the end of the first year of prescription, increasing the risk of developing diseases or experiencing adverse outcomes \cite{Burnier2018}. The lack of medication adherence has been attributed to several factors, including patients' misconceptions over their health, the cost of medication, and the complexity of treatment when multiple medications are involved \cite{Vonbank2017}. Healthcare providers use interventions to mitigate decays in adherence, where clinical trials have tested these interventions in various settings by varying cost, patient involvement, and healthcare provider involvement \cite{Kini2018}. These interventions are divided into different categories: patient education, medication management, clinical pharmacist education, cognitive behavioral training, reminders, and financial incentives \cite{Brookhart2007}, \cite{Lee2006}. 

Interventions within the categories of patient education, clinical pharmacist education, and cognitive behavioral training show the best results, where the percentage of patients that adhere increases to almost $90\%$ \cite{Kini2018}. Nonetheless, these interventions require significant time from the healthcare providers, including training. Moreover, they also need time commitments from patients. Therefore, not all patients have the opportunity to participate. On the other hand, interventions, such as reminders via text messages, phone calls, and brochures, are less expensive, and most patients are included. Unfortunately, the success of these interventions is low \cite{Kini2018} unless they are applied alongside patient education and follow-up to identify and address barriers \cite{Bingham2021}. Nevertheless, reminders via electronic pill counts are successful, where the percentage of patients adhering increases to almost $90\%$. Similar to previous interventions, these tend to be costly, as these methods use electronic pill bottles that track when patients are opening the boxes, taking medication, refilling, and giving alerts when missing doses. Additionally, to increase effectiveness, electronic pill count methods need additional support from clinicians \cite{Kini2018}.

In general, effective interventions have a limited capacity. A limited budget restricts healthcare providers regarding how many patients they can recommend for interventions to offset non-adhering behavior \cite{Kini2018}. Additionally, these interventions are most effective from a population perspective if they can be allocated to patients at the greatest risk of non-adherence \cite{Slejko2014}, \cite{Conn2017}. A variety of factors could potentially help predict future adherence, such as demographic information, clinical information, and prior adherence \cite{McQuaid2018}. However, it is unclear how to predict long-term adherence and how to update predictions as new information about the patients is obtained over time. Moreover, patient adherence varies over time and dynamically changes depending on the patient's health behavior \cite{Franklin2015}. Therefore, the primary goals of this article are as follows:

\begin{enumerate}
    \item Propose and validate dynamic prediction models that use claims data to identify when patients will stop adhering to medications. 
    \item Understand how the dynamic prediction improves the selection of patients that should be included in adherence-improving interventions subject to budget constraints.
\end{enumerate}

To accomplish these goals, we present a forecasting model using \textit{dynamic logistic regression} (DLR) to predict whether patients will experience persistent periods of non-adherence to help providers decide which patients in their panel should be recommended for adherence-improving interventions. We apply DLR in the context of primary prevention of cardiovascular disease (CVD) and compare the performance to previously proposed methods. We use CVD as an example because it is one of the leading causes of death in the United States \cite{Pool2018}, however, our approach can be easily extended to other contexts. We focus on patients' adherence to the most well-known medication for risk reduction (statins), where long-term adherence is a serious problem, with only 35\% of patients adhering one year after starting treatment \cite{Burnier2018}. Additionally, we develop a \textit{binary integer programming} (BIP) model that uses adherence predictions to select which patients should be assigned to an intervention subject to a budget constraint at each epoch. We apply our model to longitudinal data for a large cohort of patients seen in the national Veterans Affairs health system who initiate statin treatment for primary prevention of CVD \cite{Render2011}. We evaluate a prediction horizon of 1 to 5 years because we focus on long-term adherence, and statins are effective if they are regularly taken for an extended period \cite{Stockl2008}. Additionally, we create a simulation model to test recommendations from the DLR to optimize the selection of patients for interventions subject to a budget constraint. We validate the simulation model by comparing the number of CVD events without adherence-improving interventions versus the current number of CVD events in the studied population. This framework has the objective of helping providers decide which patients to recommend for adherence-improving interventions, acknowledging their budget constraints.

The rest of the paper is organized as follows. In Section \ref{sec4:literature}, we present the literature review related to adherence forecasting and patient selection models under limited resources and highlight the differences between our approach and the ones presented in the literature. In Section \ref{sec4:methods}, we present our modeling framework consisting of the DLR model, the BIP model, and the simulation environment where we test our approach. In Section \ref{sec4:results},  we apply this framework to a case study based on patients seen in the national Veterans Affairs health system. We compare our model versus current selection rules for adherence interventions. Finally, we discuss our results and summarize our main conclusions.

\section{Literature Review}\label{sec4:literature}

The most relevant research related to our work falls into the following fields: (1) adherence prediction models; (2) DLR models for healthcare; and (3) optimization models for resource allocation. This section highlights papers related to our work and briefly describes how our proposed methodology differs.

Prior literature proposes classifying patients' adherence based on pharmacy claims data using the initial fill and subsequent refill dates together with the pill count to estimate the percentage days covered (PDC). Two approaches that use PDC are commonly applied: \textit{binary classification} or classification based on the \textit{adherence trajectory}. Binary classification assigns the patient to be adherent or not adherent using multivariate logistic regression \cite{Hickson2017} or logistic regression with lag \cite{Hu2020} to predict if the patient's adherence will drop below a predefined threshold. The trajectory method classifies the patient according to predefined trajectories based on how quickly the patient goes below a predefined adherence threshold using group-based trajectory models \cite{Franklin2015} or logistic group-based trajectory models \cite{Franklin2018}. Unfortunately, these papers classify patients in predefined trajectories, which may exclude patients with different adherence behaviors. We propose to apply a DLR model, which can construct an individual adherence trajectory for each patient that learns from past medical records and sequentially updates predictions over time as new information is obtained. 

In healthcare, DLR has been used to forecast high blood pressure in children \cite{Hamoen2018}, surgical errors \cite{I.L.2007}, and blood pressure in adults\cite{Campos2019}. For trajectory methods, the literature suggests that adherence follows an autocorrelation pattern \cite{Franklin2015}. Therefore, DLR is a great candidate as this model constructs the trajectory by dynamically estimating adherence using previous estimations and including the autocorrelation effect. The literature also uses different machine learning approaches to model adherence for multiple periods, such as temporal deep learning for five years into the future \cite{Hsu2022} and random forests for the next two weeks \cite{Koesmahargyo2020}. Similarly, our model can predict how patient adherence will evolve, giving healthcare providers additional information beyond a classification method. Furthermore, our DLR model's dynamic nature can fit each patient's initial adherence trajectory and improves prediction accuracy when new information is available. Finally, the DLR model can also estimate the relationship between key factors and adherence behavior, which black box models struggle with and is needed to understand why patients stop adhering \cite{Prakash2021}. 

To develop these models, we use key factors comprising demographics, historical clinical information, and historical adherence information. For predicting adherence to statins, most of the prior literature considers covariates such as demographic characteristics, medical risk factors, and health provider-related factors \cite{Morotti2019}, \cite{Krumme2017}. We included the adherence history as a covariate, improving upon previous studies that suggest including only the last measurement of adherence \cite{Zullig2019}. Also, DLR models usually forecast short-term risk measured in days or weeks, as opposed to years in our case. For example, a recent study forecasts blood pressure on the time scale of days \cite{Campos2019} for CVD. In contrast, we forecast adherence to statins five years into the future because cholesterol tends to follow closely within 1 to 6 years for patients with an increased likelihood of having a CVD event \cite{Grundy2018}. We also include random effects in our model to acknowledge different adherence patterns among patient groups.

Regarding resource allocation, the literature focuses on two types of problems: immediate resource allocation for patients, for example, ICUs and surgery \cite{Abdalkareem2021} \cite{Kim2015}
\cite{Listorti2022} \cite{Meng2015} \cite{Wang2022}, and long-term resource allocation, such as screenings, follow-ups, and transplants \cite{Sabouri2017} \cite{Zacharias2017} \cite{Truong2015}. Depending on the application, the literature presents different methods to solve this problem, such as linear programming models \cite{Rath2015}, dynamic programming \cite{Creps2017} \cite{Sun2017}, and simulation \cite{Freeman2017}. Regarding long-term resource allocation for multiple patients and resource constraints, multiple models used dynamic approaches such as reinforcement learning \cite{Yu2019}, multi-armed bandits \cite{Negoescu2017}, restless multi-armed bandits \cite{Lee2019}. Most of these models allocate resources for one patient per epoch, differing from our problem, where the physician allocates resources for multiple patients during the same year. Other papers, which allocate resources for multiple patients, assumed a homogeneous population \cite{Deo2013}, \cite{Herlihy2021}. Regarding allocating resources for multiple patients in a heterogeneous population, studies have assumed perfect compliance to physician recommendations \cite{Ayer2019} \cite{Chen2016} \cite{ElHajj2022} \cite{Ho2019} \cite{Zayas2019}. Our approach considers resource allocation for multiple patients in a heterogeneous population with limited resources and imperfect compliance. We model compliance as adherence to medication and the success of interventions, as adherence-improving interventions are not perfect; for some patients, these interventions may not work \cite{Prakash2021}. Therefore, we consider the probability of the intervention not working as part of the decision.

Finally, studies that define resource allocation models have focused on different aspects of adherence-improving interventions. Studies have assumed cost-effectiveness approaches to define when a patient should be selected for intervention \cite{Rao2020}. Other approaches defined the optimal coinsurance rates that maximize the population's welfare and select patients that should receive better reductions \cite{Schell2019}. Studies have also focused on defining the optimal incentive rate for financial incentive interventions \cite{Suen2022} \cite{Zhang2022}. Related to intervention capacity, studies estimated the minimum capacity needed to satisfy the required adherence level in the population \cite{Ratcliffe2019} and developed techniques for patient care within the pharmacy to serve the maximum number of patients\cite{Chen2022}. Our work focuses on fixed-capacity interventions such as patient education, clinical pharmacist education, and reminders, all of which come at a cost. We propose a binary integer programming (BIP) model in which we select multiple patients in a heterogeneous population per epoch. Given the dynamic aspect of our problem, we propose an adaptive selection rule that combines the BIP and DLR models and updates the selection policy in each epoch depending on the patients' health behavior.  

In summary, our work differs from previous studies in three main ways. First, we propose a DLR model to predict the patient's long-term adherence, thus avoiding assigning patients to predefined trajectories. Second, our DLR model is updated every time new information is gathered from the patient. Third, we propose an adaptive selection rule to select patients for adherence-improving interventions that combine the BIP and DLR models. This rule selects multiple patients in a heterogeneous population and updates itself when new information is collected. We show that this adaptive selection rule reduces the number of CVD events in a population by helping identify which patients benefit the most from the interventions. 

\section{Methods} \label{sec4:methods}
This section describes our implementation of the DLR method for predicting patients' adherence and the BIP model for assigning patients to interventions. We present definitions and assumptions, describe the dynamic forecasting model with random effects, and present alternative patient selection rules for adherence-improving interventions. Further, we present a simulation model to test the forecast results. Table \ref{tab:notation4} presents the notation used throughout this paper, where vectors are highlighted in bold.

\begin{table}[htb!]
    \centering
    \caption{List of notation for the forecasting model.}
    \begin{tabular}{c|l}
    \hline
         Notation &  Description  \\
    \hline
         $\texttt{I}$ & Set of patients indexed by $i$ where $\texttt{I}\equiv\{1,...,n\}$.\\
         $\texttt{E}$ & Set of epochs indexed by $t$ where $\texttt{E}\equiv\{1,...,T\}$.\\
         $w_{it}$ & Binary response variable that represents if patient $i$ adheres or not at quarter $t$. \\
         $x_{ikt}$ & Covariate $k$, for patient $i$ in quarter $t$.\\
         $u_i$ & Random effect associated with patient $i$.\\
         $y_{it}$ & Probability that $w_{it}$ is equal to 1 given the covariates.\\
         $\hat{y}_{it}$ & Estimation of $y_{it}$.\\
         $\boldsymbol\beta_{t}$ & Vector of coefficients at time $t$ associated to variables.\\
         $\hat{\boldsymbol\beta}_{t}$ & Estimation of $\boldsymbol\beta_{t}$.\\
         $\Sigma_{\boldsymbol\beta t}$ & Correlation matrix at quarter $t$ for the coefficients $\boldsymbol\beta_t$.\\
         $q$ & Probability of success of an intervention.\\
         $r$ & Adherence effect on risk reduction.\\
         $P_{it\tau}$ & Probability that patient $i$ does not adhere for $\tau$ epochs until the end of the planning\\ & horizon if intervened in epoch $t$.\\
         $c$ & The number of available interventions at each epoch.\\
         $S_{it}$ & Binary variable where if $S_{it}=1$ then patient $i$ is selected for intervention in epoch $t$.\\         
     \end{tabular}
    \label{tab:notation4}
\end{table}

\subsection{Dynamic logistic regression} \label{sec4:dlrMod}

Let $\texttt{I}$ be the set of patients indexed by $i$ where $\texttt{I}\equiv\{1,...,n\}$ and $n$ is the total number of patients. Also, let $\texttt{E}$ be the set of epochs to forecast indexed by $t$ where $\texttt{E}\equiv\{1,...,T\}$ and $T$ is the total number of epochs. We define our epochs in quarters of the year, as the American College of Cardiology suggests that patients on statins take cholesterol tests every 3 to 12 months \cite{Grundy2018}. Additionally, statin medication is usually prescribed for 90 days at a time. We aim to predict whether patient $i$ adheres at quarter $t$; therefore, we let $w_{it}$ denote the binary response variable that represents if patient $i$ adheres or not at quarter $t$. We also define $x_{ikt}$ as the value of covariate $k$ for patient $i$ at quarter $t$. Finally, we assume we have prior data from the date of the first refill, denoted as quarter 1, until quarter $t-1$. We use this data to predict adherence for quarter $t$ onward. We begin by predicting adherence for quarter $t$ using a logistic regression model and then use this result as an input for the DLR model to forecast quarters $t+1$ onward. Then this process continues recursively until the final quarter $T$ as shown in Figure \ref{fig:dlr}. More data becomes available when patients return to an appointment or refill medication. Therefore, for a quarter $t+\tau$, if more data is available, we rerun the DLR model using data from quarters $1$ to $t+\tau$ and predict for quarters $t+\tau+1$ to $T$.

\begin{figure}[htb!]
\begin{center}
\caption{DLR flowchart. The DLR model initializes by using a logistic regression model to predict the patient adherence for quarter $t$. Then, the DLR model employs a recursive process using Newton's method to estimate the covariance matrix and the covariates for quarters $t+1$ to $T$. With the resulting covariance matrices and covariates, we predict the patient's adherence behavior from quarters $t$ to $T$.}
\includegraphics[scale=0.4]{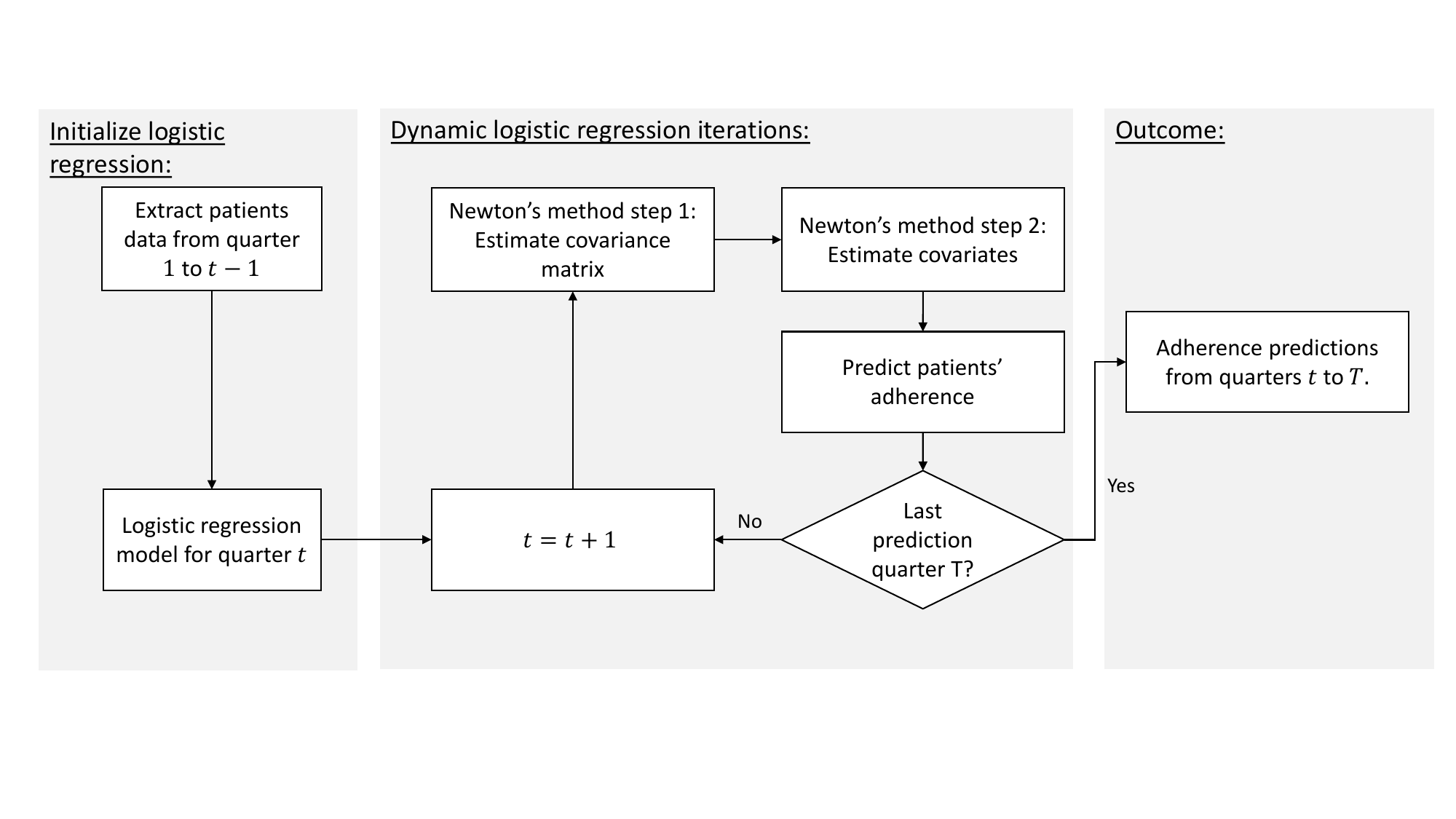}
\label{fig:dlr}
\end{center}
\end{figure}

As our goal is to estimate the probability that $w_{it}=1$, we define $y_{it}$ as the probability that $w_{it}$ is equal to 1 given the covariates. Let us define $\boldsymbol{x}_{i\cdot\cdot}$ as the vector containing the covariates history associated with patient $i$. Then, we estimate the probability $w_{it}$ is equal to 1 as   $y_{it}=P(w_{it}=1|\boldsymbol{x}_{i\cdot\cdot},\boldsymbol\beta_t)$, which is estimated by a logistic regression model, where $w_{it}=1$ represents that the patient does not adhere and $w_{it}=0$ otherwise. Notice that, for quarter $t$, $\boldsymbol{x}_{i\cdot\cdot}$ contains data from the first quarter to quarter $t-1$.  As each patient may behave differently, we define $u_i$ as the random effect associated with patient $i$. Therefore we estimate $\hat{y}_{it}$ as:

\begin{equation}\label{eqn:logit}
    \hat{y}_{it}=P(w_{it}=1|\boldsymbol{x}_{i\cdot\cdot},\boldsymbol\beta_t)=\frac{exp(u_i+\boldsymbol\beta_t \boldsymbol{x}_{i\cdot\cdot})}{1+exp(u_i+\boldsymbol\beta_t \boldsymbol{x}_{i\cdot\cdot})}.
\end{equation}

\noindent We use the data until quarter $t-1$ to forecast the adherence in $t$ by estimating the coefficients $\boldsymbol\beta_t$. The logistic regression model only predicts one quarter at a time. Therefore, we use the DLR model that starts from the standard logistic regression to iteratively generate predictions of the remainder of the time horizon $T$.  

Using the estimates of the coefficients $\hat{\boldsymbol\beta}_t$, we can forecast adherence for time $t+1$ by estimating $\boldsymbol\beta_{t+1}$. In the literature, there are several proposed approaches to estimate $\hat{\boldsymbol\beta}_{t+1}$ \cite{Mccormick2012}. In our case, we update the coefficients dynamically by defining the posterior distribution of $\boldsymbol\beta_t$, after observing the probability of non-adherence for each patient in  quarter $t$ (defined as vector $\boldsymbol y_{\cdot t}$) as follows:

\begin{equation}\label{eqn:betaup}
    p(\boldsymbol\beta_t|\boldsymbol y_{\cdot t})=p(\boldsymbol y_{\cdot t}|\boldsymbol\beta_t)p(\boldsymbol\beta_t|\boldsymbol y_{\cdot t-1}).
\end{equation}

Given the amount of data for each patient, we assume that the right-hand side of the equation is approximated by a normal distribution \cite{Penny1999}, where $\boldsymbol\beta_{t+1}|\boldsymbol y_{\cdot t} \sim N(\hat{\boldsymbol\beta}_{t}, \hat{\Sigma}_{\boldsymbol\beta_{t}})$. The vector $\hat{\boldsymbol\beta}_{t}$ is the estimation of the coefficients at the end of period $t$, and $\hat{\Sigma}_{\boldsymbol\beta_{t}}$ is the estimation of the associated covariance matrix. The exact estimation of these values is difficult, but using Newton's method, we can approximate $\hat{\boldsymbol\beta}_{t+1}$ \cite{Penny1999}:

\begin{equation}\label{eqn:betahat}
    \hat{\boldsymbol\beta}_{t+1}=\hat{\boldsymbol\beta}_{t}-\hat{\Sigma}_{\beta_{t+1}}^{-1}(\boldsymbol w_{\cdot t}-\hat{\boldsymbol y}_{\cdot t})\boldsymbol{x}_{\cdot\cdot t},
\end{equation}

\noindent where the covariance matrices are updated as follows:

\begin{equation}\label{eqn:covbeta}
    \hat{\Sigma}_{t+1}=\hat{\Sigma_t}^{-1}+\hat{\boldsymbol y}_{\cdot t}(1-\hat{\boldsymbol y}_{\cdot t})\boldsymbol{x}_{\cdot\cdot t}\boldsymbol{x}_{\cdot\cdot t}^T.
\end{equation}

\noindent Having estimated $\hat{\boldsymbol\beta}_{t+1}$, then we can estimate $\hat{y}_{it+1}$ using the updated logistic regression model. This process continues recursively until the last prediction quarter $T$ as shown in Figure \ref{fig:dlr}. 

\subsection{Patient selection model}

We formulate a BIP model to select patients for adherence-improving interventions. The model's objective is to maximize the total CVD risk reduced in the population. Given a limited intervention capacity per epoch, the objective is achieved by deciding which patients to intervene in each epoch. Therefore, in this model, we assume that: (1) the patient receives an intervention at most once over the planning horizon; if the intervention does not work, the patient is not selected again for an intervention, (2) the CVD risk is reduced if the patient is selected for intervention and the intervention works, and (3) the probability that the intervention works ($q$) and the adherence effect on relative risk reduction ($r$) are the same for all patients, where $r \in [0,1]$. 

The risk reduction per patient is estimated in terms of the initial 10-year risk for a CVD event ($CVD_i$). If a patient adheres, the new 10-year CVD risk will be $CVD_i\times(1-r)$ by the end of 1 epoch into the intervention. The DLR model outputs the probability $\hat{y}_{it}$, that patient $i$ does not adhere in epoch $t$. As there is a natural risk reduction for patients that adhere, assigning patients to interventions, who would otherwise have adhered, does not provide any additional reduction in risk. To estimate how much the CVD risk is reduced for multiple periods due to the intervention, we define $a_{it}$ as the total risk reduction throughout the planning horizon for patient $i$ if selected for intervention in epoch $t$. For example, if patient $i$ does not adhere, we let $q=1$. If that patient is selected to receive an intervention at epoch $t$, the risk after the intervention is $(1-r)^{T-t+1}CVD_i$ and $a_{it}=CVD_i-(1-r)^{T-t+1}CVD_i$. Nonetheless, as there is a probability that the patient adheres in any epoch $\tau \geq t$, estimated as $(1-\hat{y}_{i\tau})$, we incorporate this probability into estimating $a_{it}$. For example, let $t=T-2$ and $q=1$, then $a_{iT-1}=CVD_i-(1-\hat{y}_{iT-1})(1-\hat{y}_{iT})CVD_i-\hat{y}_{iT-1}(1-\hat{y}_{iT})(1-r)CVD_i-(1-\hat{y}_{iT-1})\hat{y}_{iT}(1-r)CVD_i-\hat{y}_{iT-1}\hat{y}_{iT}(1-r)^2CVD_i$. We define $P_{it\tau}$ as the probability that patient $i$ does not adhere for $\tau$ epochs until the end of the planning horizon if intervened in epoch $t$. The total expected risk reduced for patient $i$ is estimated as:

\begin{equation}\label{eqn:rewards}
a_{it}=CVD_i - q\sum_{\tau=0}^{T-t+1} P_{it\tau}(1-r)^{\tau}CVD_i.
\end{equation}

\noindent Finally, we define $S_{it}$ as a binary variable where $S_{it}=1$, if patient $i$ is selected for intervention in epoch $t$ and $S_{it}=0$ otherwise. Because the patient is only selected once for an intervention, we need to add the constraint $\sum_{t \in \texttt{E}} S_{it} \leq 1$. Given this constraint, the expected risk reduction for patient $i$ will be $\sum_{t\in \texttt{E}}a_{it}S_{it}$. As $a_{it}$ represents the expected CVD risk reduced, it provides an estimate of the reduction of the probability of having a CVD event in the next ten years. Therefore, $\sum_{t\in \texttt{E}}a_{it}S_{it}$ also represents the decreased number of expected CVD events for patient $i$ over the planning horizon $T$. We can now formulate the problem as a BIP model to maximize the expected total reduction of CVD events in the population, as follows:

      \begin{alignat}{3} \label{eqn:objFunction}
     \mbox{maximize } &\sum_{t\in \texttt{E}}\sum_{i \in \texttt{I}}  a_{it}S_{it} & \\ 
     \mbox{subject to } &  & \nonumber \\ \label{eqn:capConst}
     & \sum_{i \in \texttt{I}} S_{it} \leq c \quad \forall t \in \texttt{E}, & \\ \label{eqn:oneConst}
     & \sum_{t \in \texttt{E}} S_{it} \leq 1 \quad \forall i \in \texttt{I}, & \\  \label{eqn:varNature}
     & S_{it} \in \{0,1\} \quad \forall i \in \texttt{I}, \quad \forall t \in \texttt{E}.& 
     \end{alignat}

\noindent The constraint in Equation \ref{eqn:capConst} is associated with the intervention capacity per epoch, and Equation \ref{eqn:oneConst} enforces that patients are selected once. Recall that $a_{it}$ is estimated using the intervention and non-adherence probability. This BIP is a special case of the multiple knapsack problem, where finding the optimal solution is, in general, NP-hard \cite{Ilhan2010}. However, we show the instances we are concerned with are easy to solve, 

We want to understand the additional reduction of CVD events when selecting a patient in epoch $t$ versus waiting. Having estimated this difference, we want to define a policy that selects patients based on marginal risk reduction between epochs. Therefore, in Proposition \ref{prop:order}, we show that $a_{it}$ decreases in time and how to estimate $a_{it}$ and $a_{it-1}$. 

\begin{proposition} \label{prop:order}
For patient $i$, $a_{i1} \geq a_{i2} \geq \cdots \geq a_{iT}$. Furthermore, $a_{it-1}-a_{it}=r\hat{y}_{i(t-1)}(CVD_i-a_{it})$.
\end{proposition}
\begin{proof}:
We prove this proposition by estimating the difference between $a_{it-1}$ and $a_{it}$. Then we show that $a_{it-1} - a_{it} \geq 0$ for all $t \leq T$. For the complete proof, refer to Appendix \ref{app:prop}.
\end{proof}
 
We use Proposition \ref{prop:order} to prove Theorem \ref{theo:meansol}, which shows that the optimal policy from the BIP model is obtained by ranking patients in each epoch $t$ by $a_{it}-a_{it+1}$. This result means that patients with the highest marginal risk reduction of CVD events between epoch $t$ and $t+1$ should be prioritized first.

\begin{theorem} \label{theo:meansol}
Let $z$ be the optimal value of the BIP presented in Equations 6 to 9. If $cT < n$, then,

\[z=\sum_{i=1}^ca_{i1}+\sum_{i=c+1}^{2c}a_{i2}+\cdots+\sum_{i=(T-1)c+1}^{Tc}a_{iT}.\]

If $n \leq cT$, then let $\gamma$ be a positive integer such that $ (\gamma-1)c \leq n \leq \gamma c$. Then, 

\[z=\sum_{i=1}^ca_{i1}+\sum_{i=c+1}^{2c}a_{i2}+\cdots+\sum_{i=(\gamma-1)c+1}^nca_{i\gamma}.\]

Where patients are sorted as follows: For any particular epoch $t<T$, patient $i$ is prioritized over patient $j$ if $a_{it}-a_{it+1}\geq a_{jt}-a_{jt+1}$. For epoch $t=T$, patient $i$ is prioritized over patient $j$ if $a_{iT}\geq a_{jT}$.
\end{theorem}

\begin{proof}:
We prove the theorem by induction. We define a Markov decision process (MDP) model that represents the BIP formulation. For each epoch $t$, we use the MDP model to prove that the optimal value is reached when the patients are sorted by $a_{it}-a_{it+1}$ for all epochs $t$ to $T-1$. For epoch $T$, the patients are sorted by $a_{iT}$. For the complete proof, refer to Appendix \ref{app:theo}.
\end{proof}

The result of Theorem \ref{theo:meansol} yields an easy-to-apply policy based on ranking and selecting the patients from highest to lowest total CVD event reduction at each epoch. We will refer to this solution as the \textit{BIP policy}. It is important to note that this policy implies the decision-maker makes decisions at epoch 1, assuming no changes in the relative ordering of patients at each future epoch. This may or may not be true depending on observations over the course of the time horizon. In the next section, we propose an adaptive selection rule motivated by the BIP policy that uses the DLR model of Section \ref{sec4:dlrMod}. Additionally, we introduce alternative patient-intervention assignment decision rules with which we will compare our adaptive selection rule. 

\subsection{Patient-intervention decision rules}

In this section, we propose two rules motivated by the BIP policy and the DLR model and compare their results versus a standard rule used in practice. The \textit{Standard rule} represents a plausible approach used by healthcare providers in which they prioritize patients with higher risk first. The first rule that we propose is to prioritize patients using the BIP policy without DLR (\textit{BIP rule}), as discussed in the previous section. The second rule we propose is  prioritizing using the BIP policy and updating the adherence probabilities using the DLR model (\textit{BIP-DLR rule}). For each of these rules, we assume that in each epoch, we have access to $CVD_{it}$, defined as the CVD risk for patient $i$ in epoch $t$ as seen in the data.

We first start describing how the standard rule operates. From the pool of patients, we only consider patients for selection if their PDC is lower than $80\%$ at a given epoch and they are not on intervention. Because there is a finite capacity for the intervention, we sort patients by CVD risk. After each intervention, we update each patient's adherence depending on: (1) if the patient was intervened or not, (2) the intervention worked, and (3) the pharmacy claims data between epochs. We present the rule in Algorithm \ref{algo:int1}.

\begin{algorithm}
  \caption{Standard rule}\label{algo:int1}
  \begin{algorithmic}[1]
    \Procedure{Assignment}{}
      \State $\text{Patients} \gets \text{Patients information: adherence and CVD risk}$  
      \For{t=1 to T}
          \State $\text{Patients} = Sort(\text{Patients}) \gets \text{Sort non-adherent patients by CVD risk (decreasing)}$
          \State $int=0$
       \For{i=1 to P}
            \If {$int<c$}
                \State $S_{it}=1$ 
                \State $int=int+1$
            \Else
                \State $S_{it}=0$ 
            \EndIf
        \EndFor
        \State $\text{Patients}=Update(\text{Patients},t,S_{..}) \gets \text{Update patients adherence and risk}$  
      \EndFor
    \EndProcedure
  \end{algorithmic}
\end{algorithm}

For the first proposed rule, the BIP rule, we run the logistic regression once to estimate the adherence for future epochs 1 to T without applying the updating of the DLR model. This regression, without the updating, outputs the probabilities that will be used in the BIP formulation. Therefore we choose the patients with the highest $a_{it}-a_{it+1}$ for each epoch $t$, estimated as $r\hat{y}_{i(t-1)}(CVD_i-a_{it})$, as shown in Proposition 1. Notice that we use $CVD_{i1}$, as we only have access to the current CVD risk when we run the BIP model. Also, $\hat{y}_{it}$ represents the output probabilities from the standard regression model based on historical data. We present this rule in Algorithm \ref{algo:int2}.

\begin{algorithm}
  \caption{BIP rule}\label{algo:int2}
  \begin{algorithmic}[1]
    \Procedure{Assignment}{}
      \State $\text{Patients} \gets \text{Patients information: adherence and CVD risk}$  
      \State $\text{Prediction}=LR(\text{Patients},t)\gets \text{Logistic regression model output using historical data}$
      \State $A=\{a_{it}\} \gets \text{Estimate $a_{it}$ for all patients and for all $1\leq t \leq T$}$
      \State $\text{Patients}=Update(Patients,A) \gets \text{Update the patients information with A}$
      \For{t=1 to T}
          \State $\text{Patients} = Sort(\text{Patients}) \gets \text{Sort patients by $a_{it}-a_{it+1}$(decreasing)}$
          \State $int=0$
           \For{i=1 to P}
                \If {$int<c$}
                    \State $S_{it}=1$ 
                    \State $int=int+1$
                \Else
                    \State $S_{it}=0$ 
                \EndIf
            \EndFor
        \State $\text{Patients}=Update(\text{Patients},t,S_{..}) \gets \text{Update patients adherence and risk}$  
      \EndFor
    \EndProcedure
  \end{algorithmic}
\end{algorithm}

Finally, the BIP-DLR rule consists of an adaptive selection rule that runs the DLR model to update adherence probabilities at each epoch and then applies the policy obtained from Theorem 1. In this algorithm, we re-estimate $a_{it}$ after updating the adherence probabilities with the DLR model at each epoch. Notice that at each step $t$ we only need to update the coefficients from epoch $t$ to $T$. We present the rule in Algorithm \ref{algo:int3}.

\begin{algorithm}
  \caption{BIP-DLR rule}\label{algo:int3}
  \begin{algorithmic}[1]
    \Procedure{Assignment}{}
      \State $\text{Patients} \gets \text{Patients information: adherence and CVD risk}$  
      \For{t=1 to T}
          \State $\text{Prediction}=DLR(\text{Patients},t)\gets \text{DLR output using data until t-1}$
          \State $A_t=\{a_{i\tau}\} \gets \text{Estimate $a_{i\tau}$ for all patients and for all $t\leq \tau \leq T$}$
          \State $\text{Patients}=Update(Patients,A_t) \gets \text{Update the patients information with $A_t$}$
          \State $\text{Patients} = Sort(\text{Patients}) \gets \text{Sort patients by $a_{it}-a_{it+1}$ (decreasing)}$
          \State $int=0$
           \For{i=1 to P}
                \If {$int<c$}
                    \State $S_{it}=1$ 
                    \State $int=int+1$
                \Else
                    \State $S_{it}=0$ 
                \EndIf
            \EndFor
        \State $\text{Patients}=Update(\text{Patients},t,S_{..}) \gets \text{Update patients adherence and risk}$  
      \EndFor
    \EndProcedure
  \end{algorithmic}
\end{algorithm}

We next introduce the simulation model used to test the three intervention rules. 

\newpage

\subsection{Simulation model}

We developed a simulation model to assess the potential benefits of using the BIP-DLR rule to assign a limited number of patients to an adherence-improving intervention based on a hypothetical budget constraint. To initialize the simulation, we define year $t$ as a clinical decision maker's decision point. We divided the data into two sets, before $t$ (the past) and $t$ onward (the future). The data set before year $t$ is used to initialize a logistic regression model. We update the DLR model to predict future adherence at each future time period. We compare alternative resource allocation algorithms (described in the previous section) to assign limited interventions to patients using factors that include future predicted adherence. We continue with the assumptions presented regarding how the intervention works for this model. Nevertheless, we relax the assumption that the patient is only selected once for the intervention. For the simulation model, we assume that patients not selected for intervention or if previous interventions did not work will be available to be selected again for future periods with the same success probability $q$. This relaxation is given such that patients have future opportunities to be included in an intervention again. On the other hand, patients with a successful intervention will not be selected again, as they are assumed to adhere to medications. We continue this iterative process until the last simulated epoch $T$. This relaxation is included within the three rules presented in the previous section, where the pool of patients to select from includes patients that have not been selected for intervention or whose previous interventions did not work. We present the simulation flowchart in Figure \ref{fig:simu}.

\begin{figure}[htb!]
\begin{center}
\caption{Simulation flow. The model simulates how the patient's adherence and CVD risk behave depending on whether or not the patient is assigned to an intervention.}
\includegraphics[scale=0.45]{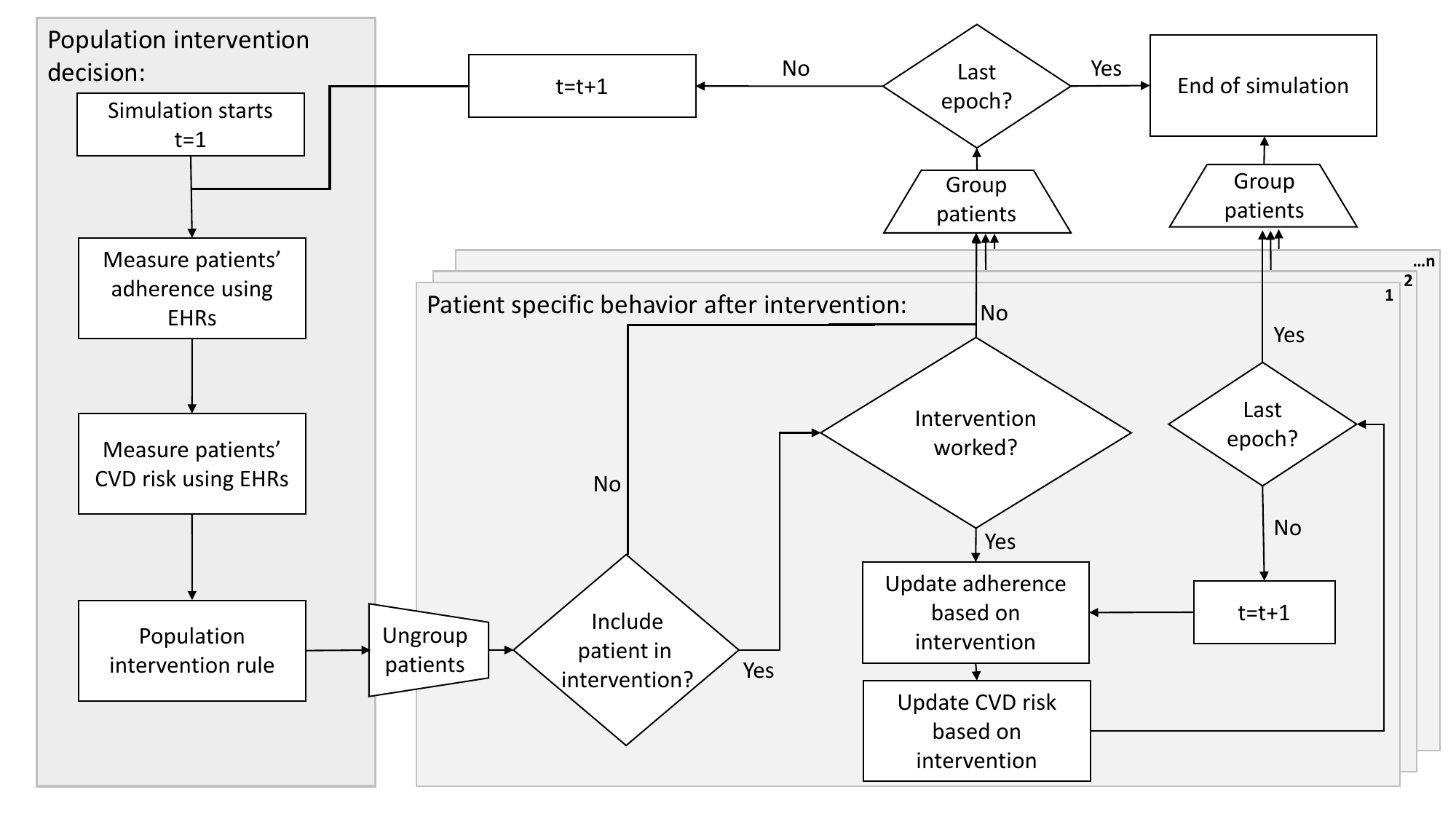}
\label{fig:simu}
\end{center}
\end{figure}

 After selecting patients for intervention and updating the adherence if the intervention is successful, the epoch is updated to $t+1$. We divide the EHR and claims data up to $t+1$ and after $t+1$, with the difference that we do not include patients that had a successful intervention. We simulate a 5-year time horizon, where each year is considered an epoch. At each iteration, time advances by one year, where we update the patients' risk using the EHR data if the patient is not on an intervention or update the risk using the risk reduction parameter if the patient is on an intervention and the intervention worked. To test the effectiveness of the intervention rules, we estimate the reduction in the number of CVD events after the simulated intervention horizon. To estimate the number of events, we calculate the absolute 10-year CVD risk for each patient after the simulation and the number of CVD events after ten years. We compare the model's results versus when no adherence-improving intervention exists.

\section{Results} \label{sec4:results}

This section presents the case study, the results of the DLR model, the patient-intervention decision rules, and the simulation model. We initialize the DLR model by estimating a standard logistic regression model on a training set to calibrate our model, then estimate the prediction for the studied time horizon, and finally simulate the patients' CVD risk with and without the prediction.  

\subsection{Case study}

We used a longitudinal data set of randomly selected patients seen in the national Veterans Affairs (VA) health system between 2003 and 2018. The data set was based on a cohort of 10,000 patients sampled from the national VA population. We included patients with at least one year on statins, without diabetes, and without any previously diagnosed heart disease or related conditions. Using these selection criteria yielded 3753 patients with data between 2003 to 2018, between the ages of 40 to 80 years old. 

Covariates included demographics, medical health factors, blood pressure and cholesterol levels, and observation dates. Within demographics, we consider sex, race, smoking habits, and age. The smoking status was assumed to be constant throughout the study, as we only had access to the patient's smoking status at the beginning of the study. We considered low-density lipoprotein cholesterol (LDL), total cholesterol, and systolic blood pressure (SBP) as medical health factors. We considered the number and dates of cholesterol tests and blood pressure tests. Finally, we used pharmacy claims data (start date, pill quantity, and refill dates) to compute the PDC, estimated as the percentage of the number of days with medication over the total number of days between refills quarterly \cite{Burnier2018}. We estimated the PDC quarterly as the American College of Cardiology suggests that patients on treatment be seen at least every three months\cite{Grundy2018}. Using a common assumption in the literature, we defined a patient as adherent in a given quarter if their PDC was greater than or equal to 80\%; otherwise, they were considered non-adherent \cite{Burnier2018}. 

\begin{table}[htbp]
    \centering
        \caption{Baseline characteristics of the population. We divide the patients based on their lifetime adherence behavior.}
    \begin{tabular}{lccc}
         \hline
         &  \multicolumn{2}{c}{Mean ($\pm$ SD) or No.(\%)} &\\\cline{2-3}
         Characteristic & Non-adherent patients & Adherent patients & P\\
         \hline
        n & 1569& 2095&\\
        Sex & & & 0.001 \\
        \quad Male & 1430 (91.1\%) & 1973(94.2\%) & \\
        \quad Female &     139 (8.9\%) &    122 (5.8\%) &  \\     
        Race & & & $<$0.001\\
        \quad White & 1196 (76.2\%) & 1842 (87.9\%) & \\
        \quad Black & 373 (23.8\%) & 253 (12.1\%) & \\ Smoking Habits & & & 0.026\\ 
        \quad Smoker & 372 (23.7\%) &  431 (20.6\%) &\\
        \quad Non-smoker & 1197 (76.3\%) &  1664 (79.4\%) &\\
        Age & 59.22 (11.76) & 63.84 (11.05) & $<$0.001\\    
        Percentage Days Cover & 63\%  (15\%) & 90\% (5\%) &  0.79\\
        Blood Pressure & & &\\
        \quad SBP & 129.90 (9.71) & 128.99 (8.72) & 0.003\\     
        \quad Number of tests & 34.03 (33.54) & 42.40 (47.16) & $<$0.001\\
       Cholesterol & & &\\
       \quad LDL & 123.06 (28.76) & 103.54 (23.60) &   $<$0.001\\     
       \quad Total Cholesterol & 199.09 (33.13) & 177.08 (29.92) &  $<$0.001\\     
       \quad Number of tests & 10.99 (7.11) & 14.22 (7.25) & $<$0.001\\ 
  \hline
    \end{tabular}
    \label{tab:tabone}\\

\end{table}

We present an overview of the study population in Table \ref{tab:tabone}. Most of the population is comprised of white male patients (70\% of the population), followed by black men (20\%), then white women (6.5\%), and finally, black women (2\%). Within the data, other races represent less than 1\% of the data points and are not differentiated. For this reason, we focused on the four patient groups mentioned before. The sex, race, and smoking habits data represent the number of patients with each characteristic. For age, PDC over the time horizon, blood pressure, and cholesterol, we estimated the mean and standard deviation of the population for adherent and non-adherent patients. Finally, the number of cholesterol and blood pressure tests represents the average per patient of the number of tests they take throughout the study period of 5 years. In summary, Table \ref{tab:tabone} shows that non-adherent patients tend to be younger, have fewer cholesterol and blood pressure tests, and have higher LDL and total cholesterol.    

\subsection{Logistic regression initialization and covariates}

The initial logistic regression at the start of the time horizon comprises three types of covariates: random effects, categorical covariates, and numerical covariates. The random effects are represented within the intercept term, which is significant regarding the probability that a patient does not adhere to medication. This denoted intrinsic behaviors regarding each patient, which are not evaluated by any other variables. Regarding past adherence, the study includes patients with at least one year on medication. Therefore, there is always enough history to predict future adherence.

\begin{table}[htbp]
    \centering
        \caption{Coefficients for the logistic regression for historical data. We also present the variance inflation factors (VIF) to test for multicollinearity.}
    \begin{tabular}{lcccc}
         \hline
         Covariate & Estimate & Std. Error & P-value & VIF\\
         \hline
        Intercept & 3.449 & 0.545 & $<$0.001* & \\
        Sex & -0.285 & 0.195 & 0.144 & 1.04\\
        Race & -0.324 & 0.121 & 0.013* & 1.03\\
        Smoking habits & 0.062 & 0.100 & 0.541 & 1.06\\
        Age & $4.22\times 10^{-3}$ & $3.91 \times 10^{-3}$  & 0.280 & 1.24\\    
        Blood Pressure & & & &\\
        \quad SBP & $7.69 \times 10^{-5}$ & $2.77 \times 10^{-3}$ & 0.978 & 1.03\\     
        \quad Number of tests & 0.105 & 0.033 & 0.001* & 1.06\\
       Cholesterol & & & &\\
       \quad LDL & $5.55 \times 10^{-3}$ & $2.42 \times 10^{-3}$ & 0.022* & 4.22\\     
       \quad Total Cholesterol & $1.37 \times 10^{-3}$ & $2.05 \times 10^{-3}$ &  0.51 & 4.29\\     
       \quad Number of tests & -0.093 & 0.071 & 0.191 & 1.14\\
       Percentage days cover & & & &\\
       \quad 1st Quarter & -0.755 &0.136 &$<$0.001* & 2.71\\
       \quad 2nd Quarter & -0.707 & 0.149 &$<$0.001* & 3.06\\
       \quad 3rd Quarter & -0.190 & 0.149 & 0.204 & 3.11\\
       \quad 4th Quarter & -0.420& 0.152 & 0.006* & 3.09\\
       \quad 5th Quarter & -0.296&0.151 &0.049* & 3.02\\
       \quad 6th Quarter & -0.091&0.148 &0.541 & 2.94\\
       \quad 7th Quarter & -0.221& 0.146&0.130 & 2.78\\
       \quad 8th Quarter & -0.033&0.133 &0.807 & 2.30\\
       
  \hline
    \end{tabular}
    \label{tab:baseline4}\\

\end{table}

Table \ref{tab:baseline4} presents the coefficients of the logistic regression model. Additionally, we calculated the variance inflation factors (VIF) to test for multicollinearity. We noticed that cholesterol LDL and total cholesterol have a VIF greater than 4, and the PDC also has larger VIF values, between 2.3 and 3.1. We further estimated the correlation between the coefficients (see Figure \ref{fig:corr} in Appendix \ref{app:graphs}). The high VIF values for LDL and total cholesterol are due to the high correlation (0.87) between these covariates. This correlation is explained as the total cholesterol is estimated from LDL via Friedewald's equation \cite{Friedewald1972}. On the other hand, the high correlation between the different quarters of PDC shows an autocorrelation behavior for adherence, between 0.57 and 0.74. Notably, the correlation decreases with respect to time between PDC estimates, indicating that the value of PDC history decreases over time. We can look at these results by studying which covariates are significant to forecasting future adherence.

Regarding the categorical variables, Table \ref{tab:baseline4} shows race is the only variable that affects the probability of not adhering. From our analysis, we note that black patients are associated with lower adherence to statins which has been observed in other studies as well \cite{Statistics2018} \cite{Kenik2014} \cite{Thomson2019}.

Other statistically significant numerical covariates (based on a p-value threshold of 0.05) are the number of blood pressure tests from the beginning of the study until the current refill date, the LDL, and past percentage days covered (PDC) were significant variables related to not adhering to medication. There is a positive correlation between a higher LDL with non-adhering behavior. On the other hand, when patients take more blood pressure tests, they are more likely to stop adhering. Previous studies suggest this may be related to patients taking less medication when their test results show that their health has improved \cite{Burnier2018}. Finally, we noticed that the patient's PDC behavior from the last year and a half dramatically affects the probability of not adhering, whereas, at a high previous PDC, the probability of not adhering to medication decreases compared to a lower previous PDC. It is worth noting that prior PDC from more than a year and a half in the past has less effect on present adherence. This result suggests that a moving time window with a year and a half of prior history is sufficient to predict the likelihood that a patient will stop adhering.

\subsection{Dynamic logistic regression}

After estimating the initial logistic regression covariates, we run the DLR model to forecast the probability that patients do not adhere over the next five years. We needed sufficient data to include the patients' random effects in the model. From the data set, which comprises data between 2003 to 2018, we assigned observational data from 2009 and before to be the training set. Therefore, the testing set consisted of data from 2010 onwards, so we have sufficient data for the 5-year rolling time horizon. We define a patient as not adhering if, at least for two quarters within a particular year, the patient does not adhere. Additionally, we test the DLR model using different values for the adherence threshold. We test the model using the base case threshold of $80\%$ for adherence and estimate the AUC, as shown in Figure \ref{fig:roc}. We perform sensitivity analysis with respect to this base case, as shown in Appendix \ref{app:graphs} Figure \ref{fig:rocall}. We performed a 3-fold cross-validation to validate the results, dividing the patients into three groups and estimating the average AUC. We present these results as the healthcare provider may use a different threshold based on the evaluated panel \cite{Burnier2018}.  

\begin{figure}[htbp!]
\begin{center}
\caption{ROC and AUC curves. We present the ROC and AUC for each of the five forecasting years for an adherence threshold of 0.8}
\includegraphics[scale=0.28]{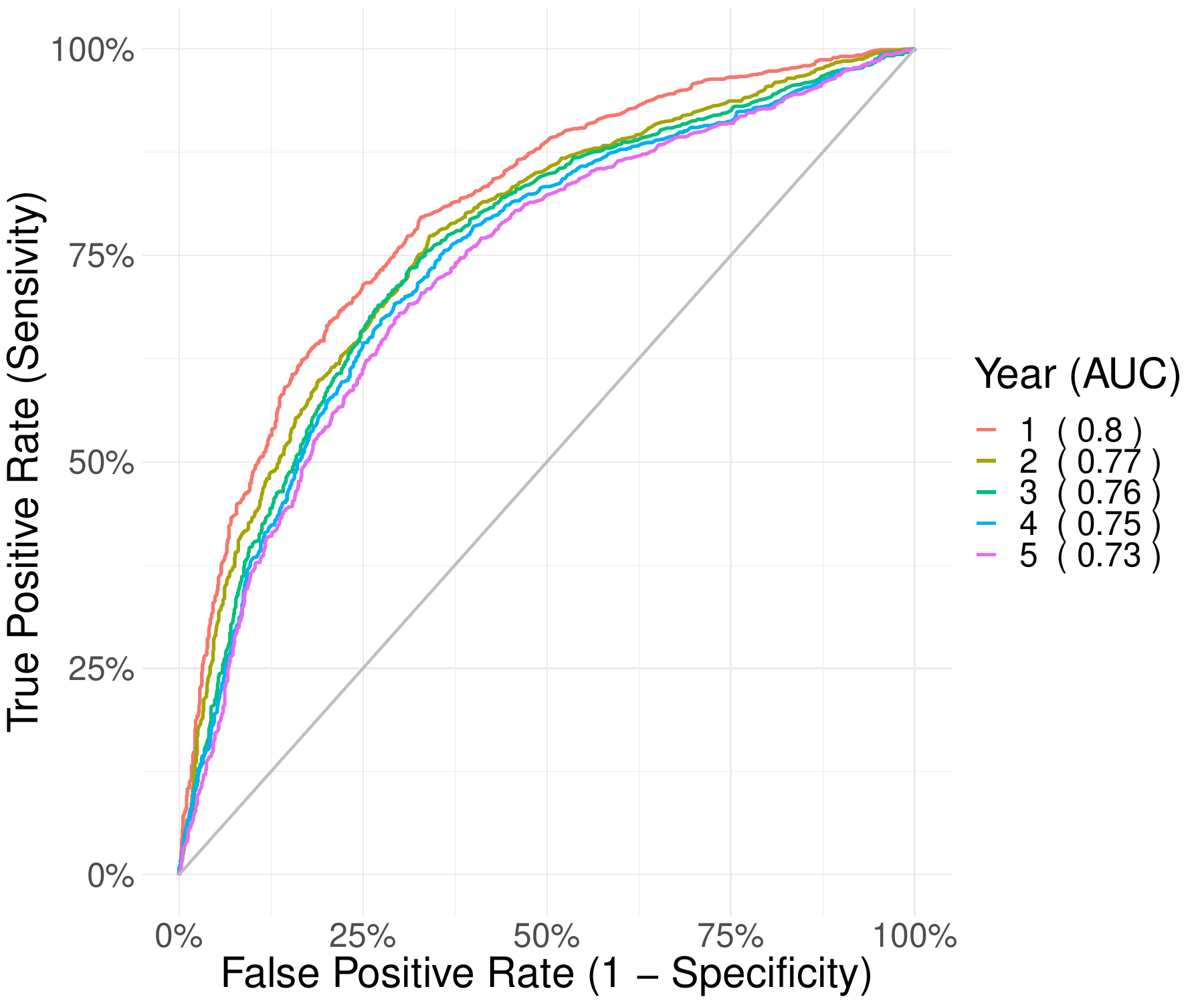}
\label{fig:roc}
\end{center}
\end{figure}

For the first year, the average AUC ranges from 79\% to 84\%, depending on the threshold chosen. The AUC decreases when forecasting for more years into the future, whereas, for the fifth year, the AUC ranges from 71\% to 74\%. 

We also measured the predictive performance of the model. We assume the decision-maker desires the logistic regression threshold that maximizes the model's accuracy. Therefore, we estimated the percentage of true positives, true negatives, false positives, and false negatives for the population, as shown in Figure \ref{fig:performance}.

\begin{figure}[htbp!]
\begin{center}
\caption{Forecast performance measurements by forecast year. For each forecast year, we present the false negative, false positive, true negative, and true positive percentages. In summary, the accuracy decreases when forecasting a later year.}
\includegraphics[scale=0.35]{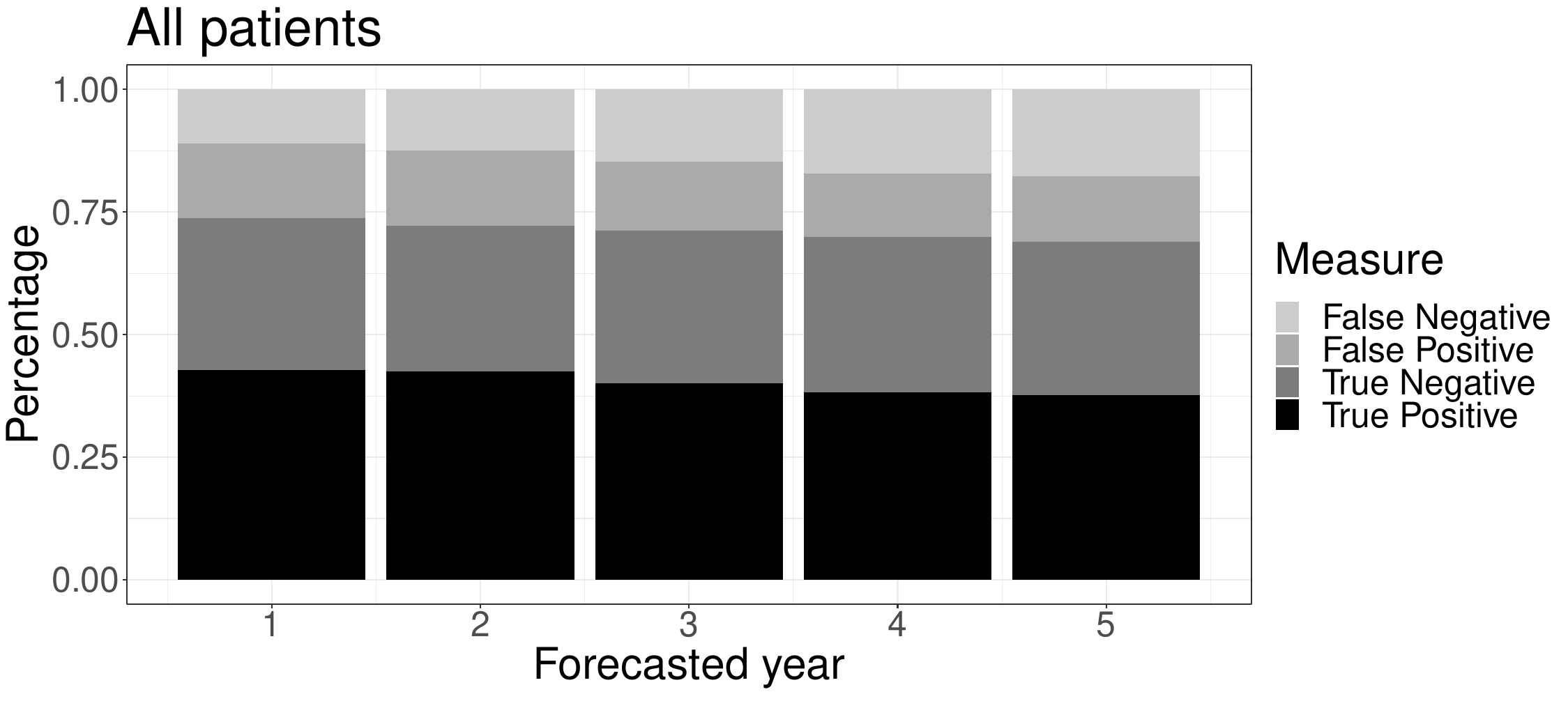}
\label{fig:performance}
\end{center}
\end{figure}

The model correctly forecasts 75\% of the patients for the first year and decreases to 70\% for the fifth year. This result shows that the model accuracy decreases when forecasting more years into the future. 

As the DLR model is intended to help healthcare providers decide which patients to include in an intervention, we focus on the percentage of false negatives. False negatives represent the portion of patients who stop adhering but for whom the model forecasts that they will adhere. Therefore, these are patients not considered for intervention. The percentage of false negatives ranges between 11\% for one year to 17\% for five years. 

\subsection{Simulation results}

We simulated the panel of patients consisting of 3753 patients from the VA population present in the case study. We apply each patient-intervention decision rule, assuming published estimates for adherence-improving interventions as shown in Table  \ref{tab:parameters4}.  As the available capacity may change depending on the intervention, we analyze how our model performs depending on different values.

\begin{table}[htbp]
    \centering
        \caption{Parameters estimates for the simulation model. The success of an intervention depends on the type of intervention provided to the patients. We present the range of each parameter when using the most common interventions.}
    \begin{tabular}{lccc}
         \hline
         Notation & Description & Value & Reference\\
         \hline
        $q$ & Probability of success of the intervention & $70\%$ to $90\%$ & \cite{Kini2018}\\
        $r$ & Adherence effect on risk reduction & $7\%$ to $12\%$ & \cite{Sussman2017}\\
    \hline
    \end{tabular}
    \label{tab:parameters4}\\

\end{table}

Additionally, we analyze the scenarios where the intervention capacity is between 15\% to 75\% of the population per year. We assume that the average capacity of 35\% of the population can be allocated and a probability of 80\% that the intervention works. We validate the model by comparing the simulation output risk without intervention, i.e., when the capacity of the intervention is 0. As there is no intervention, we update the patients based on available data. Therefore, the confidence intervals are 0, as no stochastic behavior is added to the model. The VA population average risk is 17\%, while the simulation output is $17.5\% \pm 0\%$, hence validating the model. This represents that 17,500 patients, in a population of 100,000, will have a CVD event in the next ten years if there is no adherence-improving intervention.    

We simulate the three patient-intervention decision rules, standard rule, BIP rule, and BIP-DLR rule. Additionally, we add an upper bound to the BIP-DLR rule (BIP-DLR-UB), in which we assume the probability that the intervention works is $1$. This scenario represents the best-case scenario so that we can analyze the effectiveness of the BIP-DLR algorithm. Additionally, we performed a sensitivity analysis, changing the capacity of the intervention, the probability of success, and the risk reduction. We present the sensitivity analysis in Figure \ref{fig:sen}.  

\begin{figure}[htbp!]
\begin{center}
\caption{Sensitivity analysis. We studied how the intervention capacity, the probability that the intervention works, and CVD risk reduction for patients in intervention affect each of the selection rules. We estimate the number of CVD events reduction in a population of 100,000 patients. 
}
\includegraphics[scale=0.72]{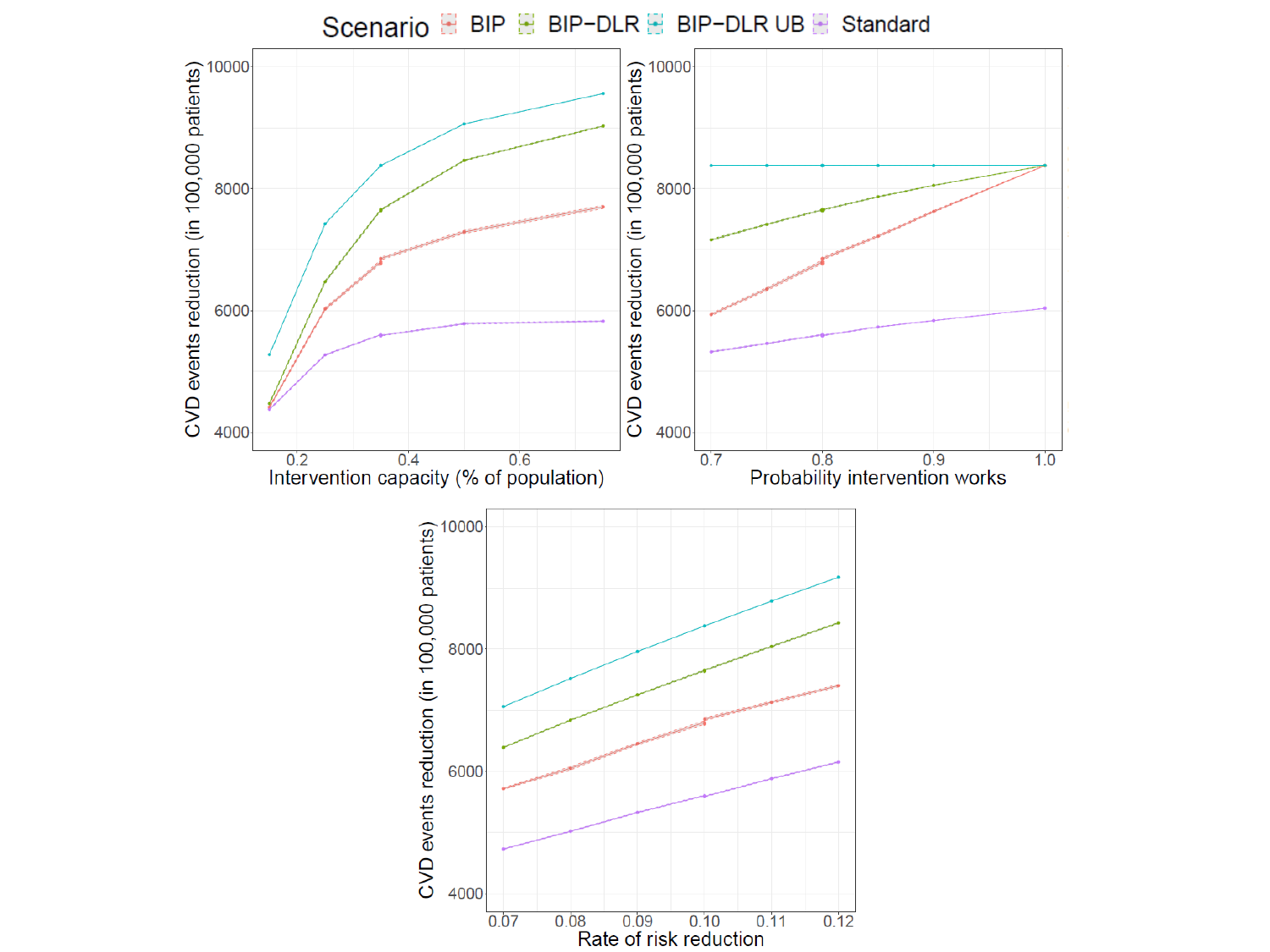}
\label{fig:sen}
\end{center}
\end{figure}

When using the base case parameters, we see that the standard rule reduces the number of CVD events by $5610 \pm 8.2$ in a population of 100,000 patients compared to not having adherence-improving interventions. Applying the BIP rule further reduces the number of CVD events by $1190\pm 28.86$ compared to the standard rule. Finally, when adding the DLR model, the BIP-DLR rule further reduces the number of CVD events by approximately 700 more, for a total of $1890\pm 12.63$ CVD events compared to the standard rule. This result shows that the BIP-DLR rule can reduce the number of CVD events by $36.5\%$ compared to the standard rule. 

Analyzing how the parameters affect the selection rules, we notice that there are diminishing returns from increasing intervention capacity. Additionally, at a higher capacity, there is a higher difference in terms of CVD reduction between the three rules. This result shows that selecting patients by the marginal difference in risk performs better than selecting by the patient's current CVD risk. Additionally, updating the adherence probabilities with the DLR model improves the results furthermore. With respect to the probability that the intervention succeeds, we notice that when it is closer to 1, there is no difference between using or not using the DLR model. This result is explained given that when the probability increases, the need to look into further periods in the future diminishes, and it is only necessary to look into the marginal difference in risk between the current and next epoch. Finally, we see that the risk reduction rate does not affect the difference between the rules. The performance is linear, given that the risk reduction linearly affects the CVD event reduction.

As mentioned in Section \ref{sec4:methods}, Theorem \ref{theo:meansol} proposes an algorithm for estimating the BIP optimal solution. Nevertheless, as shown in Figure \ref{fig:sen}, the difference between the BIP-DLR rule and the upper bound is less than the reduction that the rule already offers. This result acknowledges that using the BIP and DLR models together offers excellent results. We further tested the BIP-DLR rule for different adherence-improving interventions, where we used the average intervention success probability and monthly cost per patient found in the literature. We present the results in Table \ref{tab:intResults}.  

\begin{table}[htbp]
    \centering
        \caption{CVD events reduction by type of intervention when using the BIP-DLR rule. We compare different interventions by varying the intervention capacity. We estimate the intervention success probability and the monthly cost per patient from the literature.}
    \begin{tabular}{lccc|c|c}
         \hline
         & & &  \multicolumn{3}{c}{Intervention capacity} \\
         \cline{4-6}
         & &  &  25\% & 50\% & 75\%\\
         \cline{4-6}
         & Success & Cost &CVD events  &CVD events  &CVD events \\
         Intervention & probability  \cite{Kini2018} & per month & reduced (CI) & reduced (CI) & reduced (CI)\\
         \hline
         Financial incentives & 0.44 & \$ 42 \cite{Choudhry2011} & 4000 (21) & 6294 (17) & 7103 (21)\\
         Medication management & 0.51 & \$ 21.8 \cite{Armstrong2019} & 4546 (20) & 6930 (17) & 7642 (20)\\
         Cognitive behavioral therapy & 0.79 & \$62.11 \cite{Armstrong2019} & 6450 (11) & 8427 (14) & 8989 (8) \\
         Electronic pill count & 0.88 & \$32.06 \cite{Armstrong2019} &6921 (9) & 8746 (8) & 9269 (6)\\
         Pharmacist consultation & 0.89 & \$29.98 \cite{Ho2014} &6963 (7) & 8756 (7) & 9285 (5) \\
         \hline
    \end{tabular}
    \label{tab:intResults}\\
\end{table}

As we do not have an estimate of the intervention capacity, we tested different capacity percentages for each type of intervention. Additionally, we used the base case value of 0.1 for the risk reduction. Interventions within financial incentives and medication management tend to perform worse than other interventions, where the success probability is lower or equal to 0.51. Furthermore, interventions such as electronic pill count and pharmacist consultation perform almost similarly at a capacity of 25\% compared to the previous ones at a capacity of 75\%. Moreover, the monthly cost of financial incentives is higher than more successful interventions. Finally, the monthly cost of cognitive behavioral therapy interventions is double that of electronic pill count and pharmacist consultations, having a lower success probability.

\section{Discussion} \label{sec4:discusion}

We developed a DLR model that forecasts adherence for multiple years into the future. Our model determined the probability that the patient would stop adhering. Compared to the prior literature discussed in Section \ref{sec4:literature}, our approach further relaxed the assumption that adherence follows a specific behavior from a pool of defined trajectories. We incorporated random effects into our model, including patients' healthcare background, lifestyle, and access to healthcare, which is novel in the adherence forecasting literature. We also included the adherence history as a covariate, improving from previous studies that suggest including only the last measurement of adherence \cite{Zullig2019}. We found that including a year and a half of past adherence behavior, between 4 to 6 measurements, has a low prediction error, and including more history does not significantly improve the accuracy. Therefore, healthcare systems only need to store a year and a half of data to use this forecasting approach, reducing computing time and data requirements. 

We tested the DLR model with longitudinal data for cholesterol and blood pressure of a cohort of 10,000 randomly selected patients seen in the Veteran Affairs health system. The AUC ranged from 79\% to 84\% for the first year, depending on the PDC threshold chosen. Then, the AUC range decreased year by year until getting from 71\% to 74\% for the fifth year, depending on the PDC threshold chosen. Our model false positive measure tends to remain constant when predicting different years into the future. Therefore, the error of choosing patients for intervention when they do not need it is always relatively low, increasing the chances that patients that benefit the most from the interventions are chosen first. 

We developed a BIP model that selects patients from a heterogeneous population for interventions. Additionally, we estimated the marginal risk reduced when selecting a patient for intervention versus waiting for the next intervention cycle. We proved that the BIP optimal policy is equivalent to sorting patients by the marginal risk reduced and selecting patients from highest to lowest until the intervention capacity is depleted. Thus, the BIP has a simple polynomial time solution for this special case of the multiple knapsack problem. 

We simulated the application of DLR within the BIP when deciding which patients to include in an increase-adherence intervention. Applying the adaptive heuristic that includes the BIP and DLR models reduces the number of CVD events by $36.5\%$ compared to a base case rule of sorting patients by their current risk. Additionally, we compared different types of interventions and analyzed their performance at different intervention capacities. We recommend applying interventions that consider electronic pill counts and pharmacist consultations, as their performance is higher than other interventions and at a lower cost. Finally, these results show that adaptive approaches successfully capture adherence dynamics that fluctuate over time. 

Despite our best efforts, there are notable limitations to our study. The VA data set corresponds to patients with better access to healthcare, tend to have more frequent follow-ups, and have higher medication adherence \cite{Statistics2018}. Additionally, the VA population primarily corresponds to white and black male patients; unfortunately, we cannot access information about other races or ethnicities. Nevertheless, our model shows that the forecast has similar outcomes to other demographic groups. Therefore, a population with less access to healthcare will benefit from a more significant risk reduction, as the intervention will improve health benefits. Additionally, we understand that using PDC to estimate adherence implies the assumption that patients take all the refilled medication. Also, intervention failure probabilities are considered one size fits all, which may not be true in practice. While this may be reasonable for a population study such as ours, there could be future opportunities to consider tailoring interventions based on individual patient risk factors if such data becomes available. Some studies suggest that questionnaires, constant surveillance, and electronic pill counts are the most effective to measure adherence \cite{Shi2010} \cite{Zeller2008}. However, gathering information for these methods is costly \cite{Cherry2009}, especially in the context of studies, such as ours, that utilize longitudinal data. 

These limitations notwithstanding, our study offers a foundation for future extensions. Future work might consider improving the patient selection model by defining a dynamic stochastic model that includes the DLR prediction and the patients' health stochastic behavior. Our adaptive selection rule provides a starting point to investigate some of these approaches and further decrease the number of CVD events in the population. Additionally, we understand that our studied population has certain characteristics that differ from the general population. We propose to analyze the model's accuracy with a different population mix as a natural next step. Finally, our model assumes that the probability of the intervention working and the adherence effect on risk reduction is the same for all patients. The surveyed studies on adherence interventions analyze the overall effects of successful interventions on patients. Therefore, we recommend updating the patient selection model with studies that analyze the differences between patients regarding intervention success. Testing the model with data that has constant adherence follow-up or electronic pill count may increase the forecast accuracy. Finally, our approach could find applications in multiple diseases that rely on adherence to medication for long periods.

\section{Conclusions} \label{conclusions}

After the first year on medication, on average, 50\% of the patients will adhere, increasing the risk of cardiovascular diseases. For many patients, adherence decreases over time, and healthcare providers must identify when patients will stop adhering to prevent further complications. Different adherence-improving interventions exist, where the most efficient ones are also the most expensive and thus likely limited in their availability. Therefore, healthcare providers need to be selective in choosing patients for intervention to mitigate future low adherence. Our study shows that using DLR models with random effects can potentially alert healthcare providers to which patients will stop adhering. Additionally, the forecast supports decisions associated with selecting patients for adherence-improving interventions. Thus, we propose an adaptive selection rule that combines a BIP model with the DLR model to select patients. We prove, under sufficient conditions, that to minimize the number of CVD events in the population, healthcare providers should select patients with the highest marginal risk reduced. Our case study, based on a large population of patients seen in the VA, suggests our adaptive selection rule may significantly reduce the number of CVD events.

\appendix
\section{Appendix}

\subsection{Proof for Proposition 1} \label{app:prop}

For any epoch $t$, $P_{it0}=(1-\hat{y}_{it})(1-\hat{y}_{it+1})\cdots(1-\hat{y}_{iT})$ and $P_{i(t-1)0}=(1-\hat{y}_{it-1})(1-\hat{y}_{it})\cdots(1-\hat{y}_{iT})$. Then, $P_{i(t-1)0}=(1-\hat{y}_{it-1})P_{it0}$. For any $\tau >0$ and $\tau <T-t+2$, $P_{i(t-1)\tau}=P_{it\tau}(1-\hat{y}_{it-1})+\hat{y}_{it-1}P_{it(\tau-1)}$, as $P_{it\tau}(1-\hat{y}_{it-1})$ represents the probability that in $\tau$ epochs the patient is non-adherent between $t$ and $T-t+2$ and is adherent en $t-1$, and  $\hat{y}_{it-1}P_{it(\tau-1)}$ represents the patient is non-adherent in $\tau-1$ epochs between $t$ and $T-t+2$ and is non-adherent en $t-1$. Finally, $P_{i(t-1)(T-t+2)}=\hat{y}_{it-1}P_{it(T-t+1)}$. Then the difference between $a_{it-1}$ and $a_{it}$ is estimated as:

\begin{eqnarray*}
a_{it-1}-a_{it}&=&CVD_i - q\sum_{\tau=0}^{T-t+2} P_{it-1\tau}(1-r)^{\tau}CVD_i -CVD_i + q\sum_{\tau=0}^{T-t+1} P_{it\tau}(1-r)^{\tau}CVD_i\\
&=& q\sum_{\tau=0}^{T-t+1} P_{it\tau}(1-r)^{\tau}CVD_i - q\sum_{\tau=0}^{T-t+2} P_{it-1\tau}(1-r)^{\tau}CVD_i\\ 
&=& qCVD_i\Big(\sum_{\tau=0}^{T-t+1} P_{it\tau}(1-r)^{\tau}-\sum_{\tau=0}^{T-t+2} P_{it-1\tau}(1-r)^{\tau}\Big)\\
&=&qCVD_i\Big(P_{it0} - (1-\hat{y}_{it-1})P_{it0} - \hat{y}_{it-1}P_{it(T-t+1)}(1-r)^{T-t+2}\\ 
&+& \sum_{\tau=1}^{T-t+1} P_{it\tau}(1-r)^{\tau}- (P_{it\tau}(1-\hat{y}_{it-1})+\hat{y}_{it-1}P_{it(\tau-1)})(1-r)^{\tau}\Big)\\
&=&qCVD_i\Big(\hat{y}_{it-1}P_{it0} - \hat{y}_{it-1}P_{it(T-t+1)}(1-r)^{T-t+2}\\ 
&+& \sum_{\tau=1}^{T-t+1} P_{it\tau}\hat{y}_{it-1}(1-r)^{\tau}-\hat{y}_{it-1}P_{it(\tau-1)}(1-r)^{\tau}\Big).\\
\end{eqnarray*}

\noindent By reorganizing the terms, the difference is estimated as:

\begin{eqnarray*}
a_{it-1}-a_{it} &=&qCVD_i\Big(\hat{y}_{it-1}P_{it0} - \hat{y}_{it-1}P_{it0}(1-r) + \hat{y}_{it-1}P_{it1}(1-r) - \hat{y}_{it-1}P_{it1}(1-r)^2 \\
&+& \cdots + \hat{y}_{it-1}P_{it(T-t+1)}(1-r)^{T-t+1} - \hat{y}_{it-1}P_{it(T-t+1)}(1-r)^{T-t+2}\Big)\\
&=&qCVD_i\Big(\hat{y}_{it-1}P_{it0}(1-(1-r)) + \hat{y}_{it-1}P_{it1}(1-r)(1-(1-r)) \\
&+& \cdots + \hat{y}_{it-1}P_{it(T-t+1)}(1-r)^{T-t+1}(1-(1-r))\Big)\\
&=&qCVD_i\sum_{\tau=0}^{T-t+1}r\hat{y}_{it-1}P_{it\tau}(1-r)^{\tau}\\
&=&r\hat{y}_{it-1}q\sum_{\tau=0}^{T-t+1}P_{it\tau}(1-r)^{\tau}CVD_i\\
&=&r\hat{y}_{it-1}(CVD_i-a_{it}).\\
\end{eqnarray*}

\noindent Therefore, since $r\hat{y}_{it-1}(CVD_i-a_{it}) \geq 0$, it follows that $a_{i1} \geq a_{i2} \geq \cdots \geq a_{iT}$. $\square$

\subsection{Proof for Theorem 1} \label{app:theo}

Define a Markov decision process (MDP) model with the following components:
\begin{itemize}
    \item Epochs: Instant in time when patients are selected for interventions.
    \item States: Define $s_{it}$ such that if $s_{it}=1$ then patient $i$ was selected for intervention before $t$, and $0$ otherwise. Define $\texttt{X}_t^0$ as the set of patients who are available for interventions in epoch $t$, then if $s_{it}=0$, then $i\in \texttt{X}_t^0$. Define $\texttt{X}_t^1$ as the set of patients who are not available for interventions at the start of epoch $t$, then if $s_{it}=1$, then $i\in \texttt{X}_t^1$.
    \item Decisions: Define $\texttt{D}_t$ as the set of feasible decisions to make in epoch $t$. If $ \boldsymbol d_t\in \texttt{D}_t$, then $\boldsymbol d_t$ is a vector of size $n$ where $d_{it}=1$ if patient $i$ is selected for intervention in epoch $t$ and $0$ otherwise. Additionally, if $i\in\texttt{X}_t^1$, then $d_{it}=0$ such that $\sum_i d_{it} \leq c$ $\forall$ $i,t$.  
    \item Transitions: If in $t$, $i\in \texttt{X}_t^0$ and $d_{it}=1$, then $i\in\texttt{X}_{t+1}^1$ and $i \notin \texttt{X}_{t+1}^0$. Otherwise, patient $i$ will remain in $\texttt{X}_t^0$ until the next epoch. Then we define $\texttt{W}^0$ and $\texttt{W}^1$ as the functions that represents the sets in $t+1$, such that  $\texttt{X}_{t+1}^0=\texttt{W}^0(\texttt{X}_{t}^0,\boldsymbol d_t)$ and $\texttt{X}_{t+1}^1=\texttt{W}^1(\texttt{X}_{t}^1,\boldsymbol d_t)$.
    \item Immediate rewards: The immediate rewards of taking decision $\boldsymbol d_t$ in epoch $t$ are estimated as $\sum_{i \in \texttt{X}_t^0}a_{it}d_{it}$.
\end{itemize}

Given the components of the MDP, define $V_t(\texttt{X}_t^0,\texttt{X}_t^1)$ as the total maximal risk reduction from epoch $t$ to epoch $T$, where the Bellman's equations are as follows:

\[V_t\big(\texttt{X}_t^0,\texttt{X}_t^1\big)=\max_{\boldsymbol d_t \in D_t}\Big\{\sum_{i \in \texttt{X}_t^0}a_{it}d_{it}+V_{t+1}\big(\texttt{W}^0(\texttt{X}_{t}^0,\boldsymbol d_t),\texttt{W}^1(\texttt{X}_{t}^1,\boldsymbol d_t))\Big\} \quad \forall (\texttt{X}_t^0,\texttt{X}_t^1\big),\forall t\in \texttt{E}-\{T\},\]

\noindent where the boundary conditions are $V_T\big(\texttt{X}_T^0,\texttt{X}_T^1\big)=\max_{d_T \in D_T}\Big\{\sum_{i \in \texttt{X}_T^0}a_{iT}d_{iT}\Big\}$ for all $ (\texttt{X}_t^0,\texttt{X}_t^1\big)$ . This MDP is equivalent to the BIP in Equations 6 to 9, where the constraints are represented in the decision set. Therefore, the MDP's optimal value and policy are the same optimal value and solution of the BIP model.

Now we prove the theorem by induction. For epoch $T$, without loss of generality, the patients in $\texttt{X}_{T}^0$ can be organized such that for all $i$ and $j$ where $i \leq j$ we have that $a_{iT} \geq a_{jT}$. Then, the policy that maximizes the rewards is to choose the first $c$ patients that belong to $\texttt{X}_{T}^0$, i.e., the first $c$ patients with the highest values of $a_{iT}$.

For $t=T-1$, define $d^*_{T}$ as the optimal policy for epoch $T$ and $d^*_{T-1}$ as the optimal policy for epoch $T-1$. Given the construction of $D_T$ and $D_{T-1}$ we know that if $d_{iT-1}^* = 1$ then $ d_{iT}^*=0$ and if $ d_{iT}^*=1$ then $d_{iT-1}^* = 0$. Let $n_{0T-1}=|\texttt{X}_{T-1}^0|$, then, without loss of generality, assume that patients $1$ to $n_{0T-1}$ belong to $\texttt{X}_{T-1}^0$ and patients $n_{0T-1}+1$ to $n$ belong to $\texttt{X}_{T-1}^1$ and are organized such that if $i,j \in \texttt{X}_{T-1}^0$ and $i<j$, then $\hat{y}_{it}(CVD_i-a_{iT})\geq\hat{y}_{jt}(CVD_j-a_{jT})$. 

Let $\boldsymbol d'_{T-1}$ be such that $d'_{1T-1}=d'_{2T-1}=\cdots=d'_{cT-1}=1$ and $d'_{(c+1)(T-1)}=\cdots=d'_{nT-1}=0$. Also, let $V'_{T-1}\big(\texttt{X}_{T-1}^0,\texttt{X}_{T-1}^1\big)$ be the expected maximum total risk reduced from epoch $T-1$ when taking decision $d'_{T-1}$, then:

\[V'_{T-1}\big(\texttt{X}_t^0,\texttt{X}_t^1\big)=\sum_{i =1}^ca_{iT-1}+V_{T}\big(\texttt{W}^0(\texttt{X}_{T-1}^0,\boldsymbol d'_{T-1}),\texttt{W}^1(\texttt{X}_{T}^1,\boldsymbol d'_{T-1})).\]

\noindent Let $\boldsymbol d''_{T-1}$ such that for a particular $j\leq c$, $d''_{jT-1}=0$, $d''_{(c+1)T-1}=1$, and  for all $i \neq j,c+1$, $d''_{iT-1}=d'_{iT-1}$. Also, let $V''_{T-1}\big(\texttt{X}_{T-1}^0,\texttt{X}_{T-1}^1\big)$ be the expected maximum total risk reduced from epoch $T-1$ when taking decision $d''_{T-1}$, then:

\[V''_{T-1}\big(\texttt{X}_{T-1}^0,\texttt{X}_{T-1}^1\big)=\sum_{i =1}^{j-1}a_{iT-1}+\sum_{i =j+1}^{c+1}a_{iT-1}+V_{T}\big(\texttt{W}^0(\texttt{X}_{T-1}^0,\boldsymbol d''_{T-1}),\texttt{W}^1(\texttt{X}_{T}^1,\boldsymbol d''_{T-1})).\]

\noindent Let $\delta=V'_{T-1}\big(\texttt{X}_{T-1}^0,\texttt{X}_{T-1}^1\big)-V''_{T-1}\big(\texttt{X}_{T-1}^0,\texttt{X}_{T-1}^1\big)$, then: 

\begin{eqnarray*}
\delta&=&V'_{T-1}\big(\texttt{X}_{T-1}^0,\texttt{X}_{T-1}^1\big)-V''_{T-1}\big(\texttt{X}_{T-1}^0,\texttt{X}_{T-1}^1\big)\\
&=& a_{jT-1}-a_{c+1 T-1}+V_{T}\big(\texttt{W}^0(\texttt{X}_{T-1}^0,\boldsymbol d'_{T-1}),\texttt{W}^1(\texttt{X}_{T}^1,\boldsymbol d'_{T-1}))-V_{T}\big(\texttt{W}^0(\texttt{X}_{T-1}^0,\boldsymbol d''_{T-1}),\texttt{W}^1(\texttt{X}_{T}^1,\boldsymbol d''_{T-1}))\\
&=& a_{jT-1}-a_{c+1 T-1}+\sum_{i=c+1}^{2c}a_{iT}-a_{jT}-\sum_{i=c+2}^{2c}a_{iT}\\
&=&a_{jT-1}-a_{c+1 T-1}+a_{c+1 T}-a_{jT}.\\
\end{eqnarray*}

Using Proposition 1, given that for all $i$, $a_{iT-1}-a_{iT}=\hat{y}_{it}(CVD_i-a_{iT})$, then $\delta=\hat{y}_{jt}(CVD_j-a_{jT})-\hat{y}_{c+1t}(CVD_{c+1}-a_{c+1T})$. The patients were sorted such that if $j<c+1$, then $\delta \geq 0$ and $V'_{T-1}\big(\texttt{X}_{T-1}^0,\texttt{X}_{T-1}^1\big) \geq V''_{T-1}\big(\texttt{X}_{T-1}^0,\texttt{X}_{T-1}^1\big)$. Therefore $d'_{T-1}$ is the policy that maximizes the total risk reduced from $T-1$ to $T$. Repeating the steps for $t=1$ to $T-2$, we have that:

\[V_{1}\big(\texttt{X}_{1}^0)=\sum_{i=1}^ca_{i1}+\sum_{i=c+1}^{2c}a_{i2}+\cdots+\sum_{i=(T-1)c+1}^{Tc}a_{iT}.\]

\noindent If $n \leq cT$ then define $\gamma$ where $ (\gamma-1)c \leq n \leq \gamma c$. Then, 

\[V_{1}\big(\texttt{X}_{1}^0)=\sum_{i=1}^ca_{i1}+\sum_{i=c+1}^{2c}a_{i2}+\cdots+\sum_{i=(\gamma-1)c+1}^nca_{i\gamma}.\quad \square\] 

\subsection{Additional graphs} \label{app:graphs}

\begin{figure}[htbp!]
\begin{center}
\caption{Correlation matrix for the logistic regression.}
\includegraphics[scale=0.55]{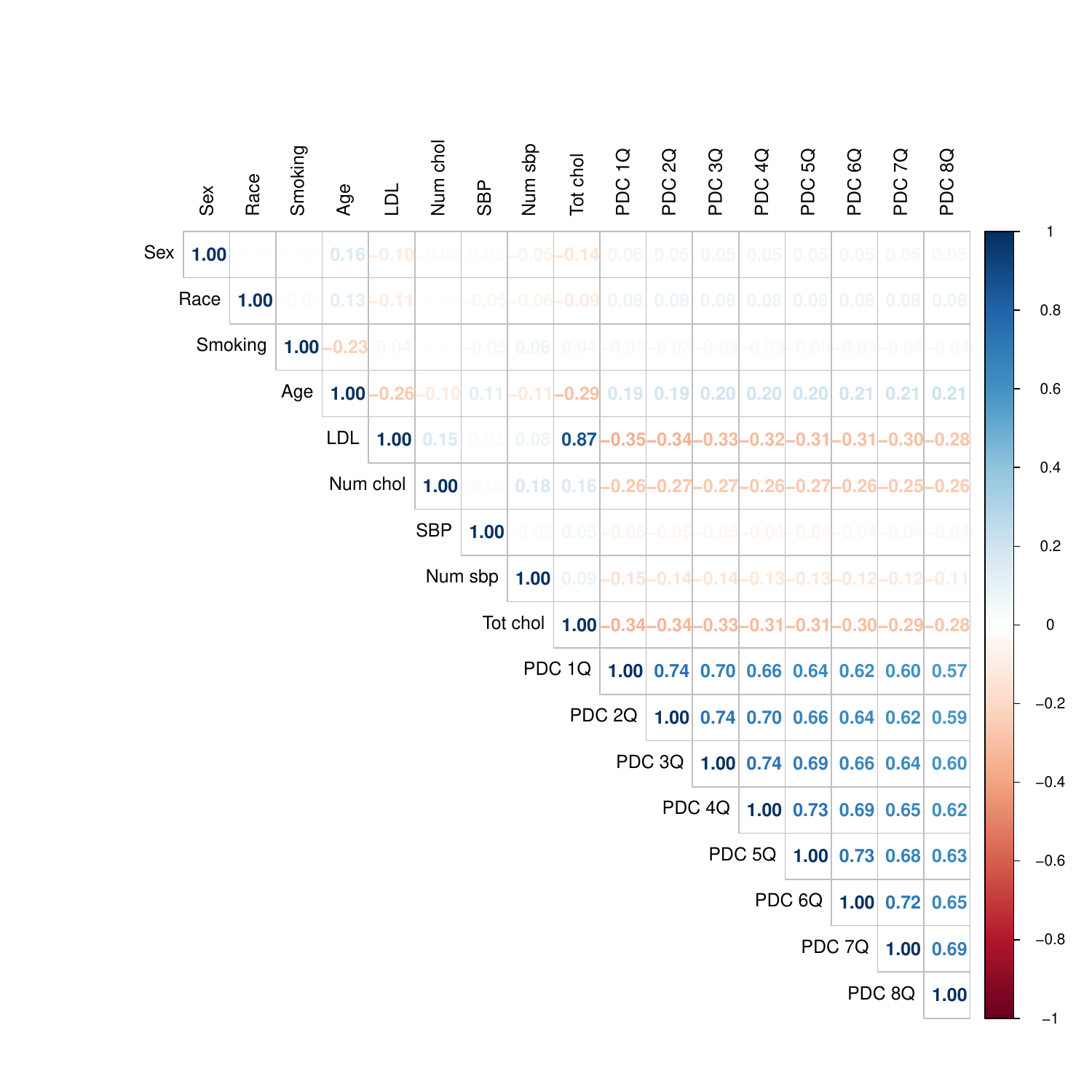}
\label{fig:corr}
\end{center}
\end{figure}
\newpage

\begin{figure}[htbp!]
\begin{center}
\caption{ROC and AUC for the each of the 5 forecasting years for adherence thresholds of 0.6 to 0.9.}
\includegraphics[scale=0.28]{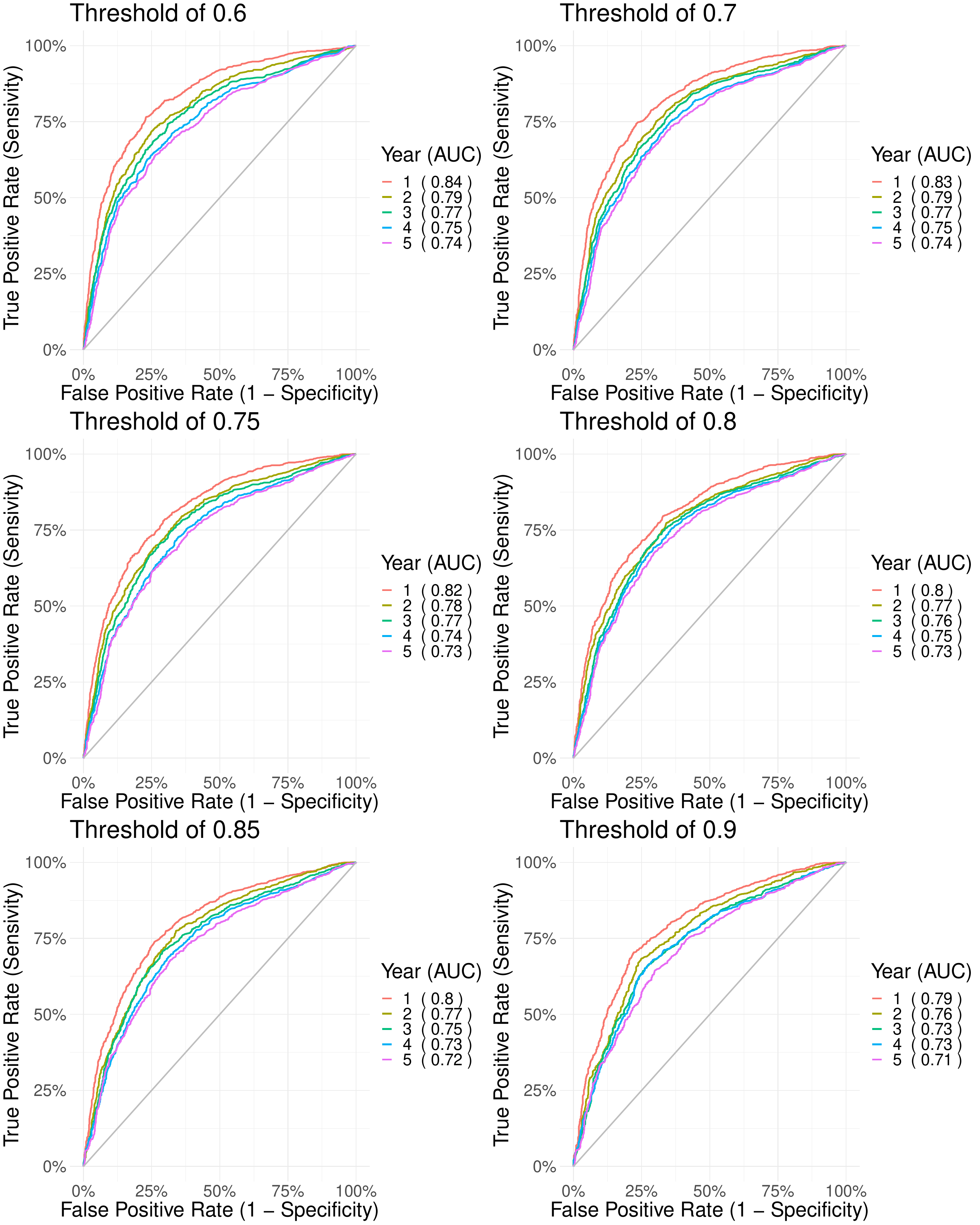}
\label{fig:rocall}
\end{center}
\end{figure}
\newpage


\begin{thebibliography}{10}

\bibitem{Abdalkareem2021}
Zahraa~A. Abdalkareem, Amiza Amir, Mohammed~Azmi Al-Betar, Phaklen Ekhan, and Abdelaziz~I. Hammouri.
\newblock Healthcare scheduling in optimization context: a review.
\newblock {\em Health and Technology}, 11:445--469, 5 2021.

\bibitem{Armstrong2019}
Shannon~O. Armstrong and Richard~A. Little.
\newblock Cost effectiveness of interventions to improve adherence to statin therapy in ascvd patients inthe united states.
\newblock {\em Patient Preference and Adherence}, 13:1375--1389, 2019.

\bibitem{Ayer2019}
Turgay Ayer, Can Zhang, Anthony Bonifonte, Anne~C. Spaulding, and Jagpreet Chhatwal.
\newblock {Prioritizing hepatitis C treatment in U.S. Prisons}.
\newblock {\em Operations Research}, 67(3):853--873, 2019.

\bibitem{Bingham2021}
Jennifer~M. Bingham, Melissa Black, Elizabeth~J. Anderson, Yawen Li, Natalie Toselli, Shawna Fox, Jennifer~R. Martin, David~R. Axon, and Armando Silva-Almodóvar.
\newblock Impact of telehealth interventions on medication adherence for patients with type 2 diabetes, hypertension, and/or dyslipidemia: A systematic review.
\newblock {\em Annals of Pharmacotherapy}, 55:637--649, 5 2021.

\bibitem{Brookhart2007}
M.~Alan Brookhart, Amanda~R. Patrick, Sebastian Schneeweiss, Jerry Avorn, Colin Dormuth, William Shrank, Boris~L.G. {Van Wijk}, Suzanne~M. Cadarette, Claire~F. Canning, and Daniel~H. Solomon.
\newblock {Physician follow-up and provider continuity are associated with long-term medication adherence: A study of the dynamics of statin use}.
\newblock {\em Archives of Internal Medicine}, 167(8):847--852, 2007.

\bibitem{Burnier2018}
Michel Burnier.
\newblock {\em {Drug Adherence in Hypertension and Cardiovascular Protection}}.
\newblock Springer, 2018.

\bibitem{Campos2019}
Luis~F Campos, Mark~E Glickman, and Kristen~B Hunter.
\newblock {Measuring effects of medication adherence on time-varying health outcomes using Bayesian dynamic linear models}.
\newblock {\em Biostatistics}, pages 1--22, 2019.

\bibitem{Chen2016}
Wei Chen, Wei Hu, Fu~Li, Jian Li, Yu~Liu, and Pinyan Lu.
\newblock {Combinatorial multi-armed bandit with general reward functions}.
\newblock {\em Advances in Neural Information Processing Systems}, (Nips):1659--1667, 2016.

\bibitem{Chen2022}
Yucheng Chen, Stephanie~A. Gernant, Charlie~M. Upton, and Manuel~A. Nunez.
\newblock Incorporating medication therapy management into community pharmacy workflows.
\newblock {\em Health Care Management Science}, 25:710--724, 12 2022.

\bibitem{Cherry2009}
Spencer~B. Cherry, Joshua~S. Benner, Mohamed~A. Hussein, Simon~S.K. Tang, and Michael~B. Nichol.
\newblock {The Clinical and economic burden of nonadherence with antihypertensive and lipid-lowering therapy in hypertensive patients}.
\newblock {\em Value in Health}, 12(4):489--497, 2009.

\bibitem{Choudhry2011}
Niteesh~K. Choudhry, Jerry Avorn, Robert~J. Glynn, Elliott~M. Antman, Sebastian Schneeweiss, Michele Toscano, Lonny Reisman, Joaquim Fernandes, Claire Spettell, Joy~L. Lee, Raisa Levin, Troyen Brennan, and William~H. Shrank.
\newblock Full coverage for preventive medications after myocardial infarction.
\newblock {\em New England Journal of Medicine}, 365:2088--2097, 2011.

\bibitem{Conn2017}
Vicki Conn and Todd Ruppar.
\newblock {Medication adherence outcomes of 771 intervention trials: systematic review and meta-analysis}.
\newblock {\em Preventive Medicine}, June(99):269--276, 2017.

\bibitem{Creps2017}
James Creps and Vahid Lotfi.
\newblock {A dynamic approach for outpatient scheduling}.
\newblock {\em Journal of Medical Economics}, 20(8):786--798, 2017.

\bibitem{Deo2013}
Sarang Deo, Kumar Rajaram, Sandeep Rath, Uday~S. Karmarkar, and Matthew Goetz.
\newblock {Planning for HIV Screening, Testing and Care at the Veterans Health Administration}.
\newblock {\em Operations Research}, 63(2):287--304, 2013.

\bibitem{ElHajj2022}
Hussein El~Hajj, Douglas~R. Bish, Ebru~K. Bish, and Hrayer Aprahamian.
\newblock Screening multi-dimensional heterogeneous populations for infectious diseases under scarce testing resources, with application to covid-19.
\newblock {\em Naval Research Logistics (NRL)}, 69(1):3--20, 2022.

\bibitem{Franklin2018}
Jessica~M. Franklin, Chandrasekar Gopalakrishnan, Alexis~A. Krumme, Karandeep Singh, James~R. Rogers, Joe Kimura, Caroline McKay, Newell~E. McElwee, and Niteesh~K. Choudhry.
\newblock {The relative benefits of claims and electronic health record data for predicting medication adherence trajectory}.
\newblock {\em American Heart Journal}, 197:153--162, 2018.

\bibitem{Franklin2015}
Jessica~M. Franklin, Alexis~A. Krumme, and William~M. Shrank.
\newblock Predicting adherence trajectory using initial patterns of medication filling.
\newblock {\em The American Journal of Managed Care}, 21(9):537--544, 2015.

\bibitem{Freeman2017}
Michael Freeman, Nicos Savva, and Stefan Scholtes.
\newblock {Gatekeepers at Work: An Empirical Analysis of a Maternity Unit}.
\newblock {\em Management Science}, 2017.

\bibitem{Friedewald1972}
W.T. Friedewald, Levy R.I., and D.S. Fredrickson.
\newblock {Estimation of the concentration of low-density lipoprotein cholesterol in plasma, without use of the preparative ultracentrifuge}.
\newblock {\em Clin Chem}, 6(18):499--502, 1972.

\bibitem{Grundy2018}
Scott~M. Grundy, Neil~J. Stone, Alison~L. Bailey, Craig Beam, Kim~K. Birtcher, Roger~S. Blumenthal, Lynne~T. Braun, Sarah de~Ferranti, Joseph Faiella-Tommasino, Daniel~E. Forman, Ronald Goldberg, Paul~A. Heidenreich, Mark~A. Hlatky, Daniel~W. Jones, Donald Lloyd-Jones, Nuria Lopez-Pajares, Chiadi~E. Ndumele, Carl~E. Orringer, Carmen~A. Peralta, Joseph~J. Saseen, Sidney~C. Smith, Laurence Sperling, Salim~S. Virani, and Joseph Yeboah.
\newblock {2018 AHA/ACC/AACVPR/AAPA/ABC/ACPM/ADA/AGS/APhA/ASPC/NLA/PCNA Guideline on the Management of Blood Cholesterol}.
\newblock {\em American College of Cardiology}, 2018.

\bibitem{Hamoen2018}
Marleen Hamoen, Yvonne Vergouwe, Alet~H. Wijga, Martijn~W. Heymans, Vincent~W.V. Jaddoe, Jos~W.R. Twisk, Hein Raat, and Marlou~L.A. {De Kroon}.
\newblock {Dynamic prediction of childhood high blood pressure in a population-based birth cohort: A model development study}.
\newblock {\em BMJ Open}, 8(11):1--11, 2018.

\bibitem{Herlihy2021}
Christine Herlihy, Aviva Prins, Aravind Srinivasan, and John~P. Dickerson.
\newblock Planning to fairly allocate: Probabilistic fairness in the restless bandit setting.
\newblock {\em ArXiv}, 2021.

\bibitem{Hickson2017}
Ryan~P. Hickson, Jennifer~G. Robinson, Izabela~E. Annis, Ley~A. Killeya-Jones, Maarit~Jaana Korhonen, Ashley~L. Cole, and Gang Fang.
\newblock {Changes in statin adherence following an acute myocardial infarction among older adults: Patient predictors and the association with follow-up with primary care providers and/or cardiologists}.
\newblock {\em Journal of the American Heart Association}, 6(10), 2017.

\bibitem{Ho2014}
P.~Michael Ho, Anne Lambert-Kerzner, Evan~P. Carey, Ibrahim~E. Fahdi, Chris~L. Bryson, S.~Dee Melnyk, Hayden~B. Bosworth, Tiffany Radcliff, Ryan Davis, Howard Mun, Jennifer Weaver, Casey Barnett, Anna Barón, and Eric J.~Del Giacco.
\newblock Multifaceted intervention to improve medication adherence and secondary prevention measures after acute coronary syndrome hospital discharge : A randomized clinical trial.
\newblock {\em JAMA Internal Medicine}, 174:186--193, 2 2014.

\bibitem{Ho2019}
Ting~Yu Ho, Shan Liu, and Zelda~B. Zabinsky.
\newblock A multi-fidelity rollout algorithm for dynamic resource allocation in population disease management.
\newblock {\em Health Care Management Science}, 22:727--755, 12 2019.

\bibitem{Hsu2022}
William Hsu, James~R. Warren, and Patricia~J. Riddle.
\newblock Medication adherence prediction through temporal modelling in cardiovascular disease management.
\newblock {\em BMC Medical Informatics and Decision Making}, 22, 12 2022.

\bibitem{Hu2020}
Feiyu Hu, Jim Warren, and Daniel~J. Exeter.
\newblock {Interrupted time series analysis on first cardiovascular disease hospitalization for adherence to lipid-lowering therapy}.
\newblock {\em Pharmacoepidemiology and Drug Safety}, 29(2):150--160, 2020.

\bibitem{I.L.2007}
Solis-Trapala I.L., J.~Carthey, V.T. Farewell, and M.R. de~Leval.
\newblock {Dynamic modelling in a study of surgical error management}.
\newblock {\em Statistics in Medicine}, 26(April):5189--5202, 2007.

\bibitem{Ilhan2010}
Taylan Ilhan, Seyed Iravani, and Mark Daskin.
\newblock The adaptive knapsack problem with stochastic rewards.
\newblock {\em Operations Research}, 59:242--248, 2020.

\bibitem{Kenik2014}
Jordan Kenik, Muriel Jean-Jacques, and Joe Feinglass.
\newblock {Explaining racial and ethnic disparities in cholesterol screening}.
\newblock {\em Preventive Medicine}, 65:65--69, 2014.

\bibitem{Kim2015}
Song~Hee Kim, Ponni Vel, Ward Whitt, and Won~Chul Cha.
\newblock {Poisson and non-Poisson properties in appointment-generated arrival processes: The case of an endocrinology clinic}.
\newblock {\em Operations Research Letters}, 43(3):247--253, 2015.

\bibitem{Kini2018}
Vinay Kini and P.~{Michael Ho}.
\newblock {Interventions to Improve Medication Adherence: A Review}.
\newblock {\em JAMA - Journal of the American Medical Association}, 320(23):2461--2473, 2018.

\bibitem{Koesmahargyo2020}
Vidya Koesmahargyo, Anzar Abbas, Li~Zhang, Lei Guan, Shaolei Feng, Vijay Yadav, and Isaac~R. Galatzer-Levy.
\newblock Accuracy of machine learning-based prediction of medication adherence in clinical research.
\newblock {\em Psychiatry Research}, 294, 12 2020.

\bibitem{Krumme2017}
Alexis~A. Krumme, Angela~Y. Tong, Claire~M. Spettell, Danielle~L. Isaman, Jessica~M. Franklin, Niteesh~K. Choudhry, Olga~S. Matlin, Troyen~A. Brennan, and William~H. Shrank.
\newblock {Predicting 1-year statin adherence among prevalent users: A retrospective cohort study}.
\newblock {\em Journal of Managed Care and Specialty Pharmacy}, 23(4):494--502, 2017.

\bibitem{Lee2019}
Elliot Lee, Mariel~S. Lavieri, and Michael Volk.
\newblock {Optimal screening for hepatocellular carcinoma: A restless bandit model}.
\newblock {\em Manufacturing and Service Operations Management}, 21(1):198--212, 2019.

\bibitem{Lee2006}
Jeannie~K. Lee, Karen~A. Grace, and Allen~J. Taylor.
\newblock {Effect of a pharmacy care program on medication adherence and persistence, blood pressure, and low-density lipoprotein cholesterol: A randomized controlled trial}.
\newblock {\em Journal of the American Medical Association}, 296(21):2563--2571, 2006.

\bibitem{Listorti2022}
Elisabetta Listorti, Arianna Alfieri, and Erica Pastore.
\newblock Hospital volume allocation: integrating decision maker and patient perspectives.
\newblock {\em Health Care Management Science}, 25:237--252, 6 2022.

\bibitem{Mccormick2012}
Tyler~H Mccormick, Adrian~E Raftery, David Madigan, Randall~S Burd, Tyler~H Mccormick, Adrian~E Raftery, David Madigan, and Randall~S Burd.
\newblock {Dynamic Logistic Regression and Dynamic Model Averaging for Binary Classification}.
\newblock {\em International Biometric Society}, 68(1):23--30, 2012.

\bibitem{McQuaid2018}
Elizabeth~L. McQuaid and Wendy Landier.
\newblock {Cultural Issues in Medication Adherence: Disparities and Directions}.
\newblock {\em Journal of General Internal Medicine}, 33(2):200--206, 2018.

\bibitem{Meng2015}
Fanwen Meng, Jin Qi, Meilin Zhang, James Ang, Singfat Chu, and Melvyn Sim.
\newblock {A Robust Optimization Model for Managing Elective Admission in a Public Hospital}.
\newblock {\em Operations Research}, 63(6):1452--1467, 2015.

\bibitem{Morotti2019}
Kendra Morotti, Julio Lopez, Vanessa Vaupel, Arthur Swislocki, and David Siegel.
\newblock {Adherence to and Persistence With Statin Therapy in a Veteran Population}.
\newblock {\em Annals of Pharmacotherapy}, 53(1):43--49, 2019.

\bibitem{Negoescu2017}
Diana~M. Negoescu, Kostas Bimpikis, Margaret~L. Brandeau, and Dan~A. Iancu.
\newblock {Dynamic Learning of Patient Response Types: An Application to Treating Chronic Diseases}.
\newblock {\em Management Science}, 64(8):3469--3488, 2017.

\bibitem{Penny1999}
William~D. Penny and Stephen~J. Roberts.
\newblock {Dynamic logistic regression}.
\newblock {\em Proceedings of the International Joint Conference on Neural Networks}, 3(October):1562--1567, 1999.

\bibitem{Pool2018}
Lindsay~R. Pool, Hongyan Ning, John Wilkins, Donald~M. Lloyd-Jones, and Norrina~B. Allen.
\newblock {Use of Long-term Cumulative Blood Pressure in Cardiovascular Risk Prediction Models}.
\newblock {\em JAMA Cardiology}, pages 1--5, 2018.

\bibitem{Prakash2021}
Aditya~M. Prakash, Carlos Vega, Vakaramoko Diaby, and Xiang Zhong.
\newblock Multidisciplinary efforts in combating nonadherence to medication and health care interventions: Opportunities and challenges for operations researchers.
\newblock {\em IISE Transactions on Healthcare Systems Engineering}, 11:255--270, 2021.

\bibitem{Rao2020}
Isabelle Rao, Adir Shaham, Amir Yavneh, Dor Kahana, Itai Ashlagi, Margaret~L. Brandeau, and Dan Yamin.
\newblock Predicting and improving patient-level antibiotic adherence.
\newblock {\em Health Care Management Science}, 23:507--519, 12 2020.

\bibitem{Ratcliffe2019}
Aaron Ratcliffe, Ann Marucheck, and Wendell Gilland.
\newblock Regional planning model for cancer screening with imperfect patient adherence.
\newblock {\em Service Science}, 11:113--137, 6 2019.

\bibitem{Rath2015}
Sandeep Rath, Kumar Rajaram, and Aman Mahajan.
\newblock {Integrated Staff and Room Scheduling for Surgeries: Methodology and Application}.
\newblock {\em Operations Research}, (August 2018), 2015.

\bibitem{Render2011}
Marta~L. Render, Ron~W. Freyberg, Rachael Hasselbeck, Timothy~P. Hofer, Anne~E. Sales, James Deddens, Odette Levesque, and Peter~L. Almenoff.
\newblock {Infrastructure for quality transformation: Measurement and reporting in veterans administration intensive care units}.
\newblock {\em BMJ Quality and Safety}, 20(6):498--507, 2011.

\bibitem{Sabouri2017}
Alireza Sabouri, Woonghee~Tim Huh, and Steven~M. Shechter.
\newblock {Screening Strategies for Patients on the Kidney Transplant Waiting List}.
\newblock {\em Operations Research}, 65(5):1131--1146, 2017.

\bibitem{Schell2019}
Greggory~J. Schell, Gian Gabriel~P. Garcia, Mariel~S. Lavieri, Jeremy~B. Sussman, and Rodney~A. Hayward.
\newblock Optimal coinsurance rates for a heterogeneous population under inequality and resource constraints.
\newblock {\em IISE Transactions}, 51:74--91, 1 2019.

\bibitem{Shi2010}
Lizheng Shi, Jinan Liu, Vivian Fonseca, Philip Walker, Anupama Kalsekar, and Manjiri Pawaskar.
\newblock {Correlation between adherence rates measured by MEMS and self-reported questionnaires: a meta-analysis}.
\newblock {\em Health and Quality of Life Outcomes}, 8(1):99, 2010.

\bibitem{Slejko2014}
Julia~F. Slejko, P.~Michael Ho, Heather~D. Anderson, Kavita~V. Nair, Patrick~W. Sullivan, and Jonathan~D. Campbell.
\newblock {Adherence to statins in primary prevention: Yearly adherence changes and outcomes}.
\newblock {\em Journal of Managed Care Pharmacy}, 20(1):51--57, 2014.

\bibitem{Stockl2008}
Karen~M. Stockl, Daniel Tjioe, Sherry Gong, Jenni Stroup, Ann~S.M. Harada, and Heidi~C. Lew.
\newblock {Effect of an intervention to increase statin use in medicare members who qualified for a medication therapy management program}.
\newblock {\em Journal of Managed Care Pharmacy}, 14(6):532--540, 2008.

\bibitem{Suen2022}
Sze~Chuan Suen, Diana Negoescu, and Joel Goh.
\newblock Design of incentive programs for optimal medication adherence in the presence of observable consumption.
\newblock {\em Operations Research}, 70:1691--1716, 5 2022.

\bibitem{Sun2017}
Zhankun Sun, Nilay~Tanık Argon, and Serhan Ziya.
\newblock {Patient Triage and Prioritization Under Austere Conditions}.
\newblock {\em Management Science}, page mnsc.2017.2855, 2017.

\bibitem{Sussman2017}
Jeremy~B. Sussman, Wyndy~L. Wiitala, Matthew Zawistowski, Timothy~P. Hofer, Douglas Bentley, and Rodney~A. Hayward.
\newblock {The Veterans Affairs Cardiac Risk Score}.
\newblock {\em Medical Care}, 55(9):864--870, 2017.

\bibitem{Statistics2018}
\text{NCVAS}.
\newblock {Profile of Veterans: 2017}.
\newblock {\em Department of Veteran Affairs}, 2018.

\bibitem{Thomson2019}
Katie Thomson, Frances Hillier-Brown, Nick Walton, Mirza Bilaj, Clare Bambra, and Adam Todd.
\newblock {The effects of community pharmacy-delivered public health interventions on population health and health inequalities: A review of reviews}.
\newblock {\em Preventive Medicine}, 124(April):98--109, 2019.

\bibitem{Truong2015}
Van-Anh Truong.
\newblock {Optimal Advance Scheduling}.
\newblock {\em Management Science}, 2015.

\bibitem{Vonbank2017}
Alexander Vonbank, Stefan Agewall, Keld {Per Kjeldsen}, Basil~S. Lewis, Christian Torp-Pedersen, Claudio Ceconi, Christian Funck-Brentano, Juan~Carlos Kaski, Alexander Niessner, Juan Tamargo, Thomas Walther, Sven Wassmann, Giuseppe Rosano, Harald Schmidt, Christoph~H. Saely, and Heinz Drexel.
\newblock {Comprehensive efforts to increase adherence to statin therapy}.
\newblock {\em European Heart Journal}, 38(32):2473--2477, 2017.

\bibitem{Wang2022}
Lien Wang, Erik Demeulemeester, Nancy Vansteenkiste, and Frank~E. Rademakers.
\newblock On the use of partitioning for scheduling of surgeries in the inpatient surgical department.
\newblock {\em Health Care Management Science}, 25:526--550, 12 2022.

\bibitem{Yu2019}
Chao Yu, Jiming Liu, and Shamim Nemati.
\newblock {Reinforcement Learning in Healthcare: A Survey}.
\newblock {\em arXiv}, 2019.

\bibitem{Zacharias2017}
Christos Zacharias and Mor Armony.
\newblock {Joint Panel Sizing and Appointment Scheduling in Outpatient Care}.
\newblock {\em Management Science}, 2017.

\bibitem{Zayas2019}
Gabriel Zayas-Cabán, Stefanus Jasin, and Guihua Wang.
\newblock An asymptotically optimal heuristic for general nonstationary finite-horizon restless multi-armed, multi-action bandits.
\newblock {\em Advances in Applied Probability}, 51:745--772, 2019.

\bibitem{Zeller2008}
Andreas Zeller, Anne Taegtmeyer, Benedict Martina, Edouard Battegay, and Peter Tschudi.
\newblock {Physicians' Ability to Predict Patients' Adherence to Antihypertensive Medication in Primary Care}.
\newblock {\em Hypertension Research}, 31(9):1765--1771, 2008.

\bibitem{Zhang2022}
Hui Zhang, Tao Huang, and Tao Yan.
\newblock A quantitative analysis of risk-sharing agreements with patient support programs for improving medication adherence.
\newblock {\em Health Care Management Science}, 25:253--274, 6 2022.

\bibitem{Zullig2019}
Leah~L. Zullig, Shelley~A. Jazowski, Tracy~Y. Wang, Anne Hellkamp, Daniel Wojdyla, Laine Thomas, Lisa Egbuonu-Davis, Anne Beal, and Hayden~B. Bosworth.
\newblock {Novel application of approaches to predicting medication adherence using medical claims data}.
\newblock {\em Health Services Research}, 54(6):1255--1262, 2019.

\end{thebibliography}
\end{document}